\documentclass[12pt]{article}

\usepackage{amsthm, amssymb,amsmath}
\usepackage[utf8]{inputenc}
\usepackage{hyperref}
\usepackage{graphicx}
\usepackage{subcaption}
\usepackage{dsfont}

\newtheorem{theorem}{Theorem}
\newtheorem{remark}{Remark}
\newtheorem{proposition}{Proposition}
\newtheorem{lemma}{Lemma}
\newtheorem{corollary}{Corollary}
\newtheorem*{definition*}{Definition}
\newtheorem*{theorem*}{Theorem}

\newcommand{\E}[1]{\mathbb{E}\left(#1\right)}
\newcommand{\Edg}[1]{\mathbb{E}_{\mathrm{dg}}\left(#1\right)}
\newcommand{\Pof}[1]{\mathbb{P}\left(#1\right)}

\renewcommand{\P}{\mathbb{P}}
\newcommand{\Pdg}{\mathbb{P}_{\mathrm{dg}}}
\newcommand{\C}{\mathbb{C}}

\newcommand{\R}{\mathbb{R}}
\newcommand{\Z}{\mathbb{Z}}

\renewcommand{\H}{\mathbb{H}}

\newcommand{\F}{\mathcal{F}}
\renewcommand{\deg}{\text{deg}}

\newcommand{\discrete}[1]{\text{Discrete}\left(#1\right)}

\newcommand{\Elambda}[1]{\mathbb{E}_{\lambda}\left(#1\right)}

\newcommand{\Plambda}{\mathbb{P}_{\lambda}}

\newcommand{\one}[1]{\mathds{1}_{#1}}

\DeclareMathOperator{\acosh}{acosh}

\newcommand{\lawequals}{\overset{(d)}{=}}
\newcommand{\FS}{\mathcal{F}_{\operatorname{sym}}}

\usepackage[top=1in, bottom=1in, left=1.3in, right=1.3in]{geometry}

\title{A Furstenberg type formula for the speed of distance stationary sequences}
\author{Mat\'ias Carrasco \and Pablo Lessa \and Elliot Paquette}

\begin{document}

\maketitle

\begin{abstract}
 We prove a formula for the speed of distance stationary random sequences.  A particular case is the classical formula for the largest Lyapunov exponent of an i.i.d.\ product of two by two matrices in terms of a stationary measure on projective space.  We apply this result to Poisson-Delaunay random walks on Riemannian symmetric spaces.    In particular, we obtain sharp estimates for the asymptotic behavior of the speed of hyperbolic Poisson-Delaunay random walks when the intensity of the Poisson point process goes to zero.  This allows us to prove that a dimension drop phenomena occurs for the harmonic measure associated to these random walks.  With the same technique we give examples of co-compact Fuchsian groups for which the harmonic measure of the simple random walk has dimension less than one.
 
 
\medskip \noindent {\bf AMS2010:} Primary 60G55, 60Dxx, 51M10, 05C81, 34D08. 
 \end{abstract}

\setcounter{tocdepth}{1}
\tableofcontents

\section{Introduction}

A random sequence is said to be distance stationary if the distribution of distances between its points is shift invariant.

The speed, or linear drift, of a distance stationary sequence is the limit 
\[\ell = \lim\limits_{n \to +\infty}\frac{1}{n}d(x_0,x_n),\]
where $d(x_0,x_n)$ is the distance between the initial point and the $n$-th point of the random sequence.

From Kingman's subadditive ergodic theorem the speed exists almost surely and in mean under the mild assumption that the expected distance between the first two points of the sequence is finite.

In Part \ref{furstenbergformulapart} of this article we prove the following integral formula for the speed of a distance stationary sequence (which we call a `Furstenberg type formula')
\[\E{\ell} = -\E{\xi(x_0)-\xi(x_1)}\]
where $\xi$ is a random horofunction depending on the past tail of the sequence (Theorem \ref{furstenbergtheorem}).

The formula is most useful when one knows that the random horofunction $\xi$ is on the horofunction boundary of the space under consideration (see Proposition \ref{corollaryboundary}).   When this is the case one can sometimes decide whether the speed is zero or positive, and obtain explicit estimates.

In the case where $(A_n)_{n \in \Z}$ is a sequence of i.i.d.\ matrices in $\text{SL}(2,\R)$ not supported on a compact subgroup, and one considers the sequence of positive parts in the polar decomposition of the products $A_n\cdots A_1$, our result implies the classical formula for the largest Lyapunov exponent in terms of a stationary probability measure on projective space originated by Furstenberg (see \cite[Chapter 2, Theorem 3.6]{bougerol-lacroix} and \cite{furstenberg1963b}).  We discuss this briefly in Section \ref{applicationssection}, together with a more elementary example.

In the rest of the article we show that Poisson-Delaunay random walks on Riemannian symmetric spaces (previously studied in \cite{benjamini-paquette-pfeffer}, \cite{paquette}, and \cite{carrascopiaggio2016}) may be fruitfully seen as distance stationary sequences.  In particular, we show that one may obtain interesting results about their speed by using the Furstenberg type formula.

For this purpose we establish, in Part \ref{stationaritypart}, that these walks are distance stationary under an appropriate bias (Theorem \ref{stationarity}).   We also prove, in Part \ref{ergodicitypart}, an ergodicity theorem (Theorem \ref{ergodicity}) for Poisson-Delaunay random walks which implies in particular that their graph speed and ambient speed are almost surely constant.  In Part \ref{speedcomparisonpart} we show show that the graph speed and ambient speed are simultaneously either positive or zero (Proposition \ref{graphvsambientspeed}).  

In Part \ref{speedasymptoticspart} of the article we prove a sharp estimate for the ambient speed of the Poisson-Delaunay random walk in hyperbolic space when the intensity of the Poisson point process is small (Theorem \ref{speedestimatestheorem}).

From this approach we obtain an alternative proof that the graph speed of hyperbolic Poisson-Delaunay random walks is positive for small intensities (Corollary \ref{speedcorollary}).  Positivity of the graph speed was proved in the two dimensional case for all intensities in \cite{benjamini-paquette-pfeffer}, by showing that the graph satisfies anchored expansion.  For all non-compact type Riemannian symmetric spaces, positivity of the speed of Poisson-Delaunay random walks was established in \cite{paquette} by using the theory of unimodular random graphs and invariant non-amenability.

We also prove that the graph speed of hyperbolic Poisson-Delaunay random walks goes to its maximal possible value $1$, when the intensity of the Poisson process goes to zero (Corollary \ref{graphspeedcorollary}).  This was previously unknown, and answers a question posed in \cite{benjamini-paquette-pfeffer}.

In Part \ref{dimensiondroppart} of the article we discuss dimension drop phenomena.  In particular we use the Furstenberg type formula to estimate the speed of certain simple random walks on co-compact Fuchsian groups well enough to show that the dimension of their harmonic measure is less than one.  We also prove, using the previously obtained estimate for speed, that for hyperbolic Poisson-Delaunay random walks with low enough intensity the dimension drop phenomena also occurs.  Both of these results are new as far as the authors are aware.

The Furstenberg type formula for speed we prove in this article (Theorem \ref{furstenbergtheorem}) is related to previous work of Gouëzel, Karlsson, and Ledrappier (see \cite{karlsson-ledrappier} and \cite{karlsson-goezel}).  We will discuss this in detail in the introduction to Part \ref{furstenbergformulapart}.

\part{A Furstenberg type formula for speed\label{furstenbergformulapart}}

The purpose of this part of the article is to construct, given a distance stationary random sequence $(x_n)_{n \in \Z}$, a random horofunction $\xi$ which captures its linear drift, and whose increments along the sequence are stationary (see below for definitions).

We do so under the hypothesis that there exists a random variable $u$ which is uniform on $[0,1]$ and independent from the given random sequence $(x_n)_{n \in \Z}$.

Of course, such a random variable can be constructed if one extends the probability space on which the random sequence $(x_n)_{n \in \Z}$ is defined.  And, any conclusion about the sequence stated only in terms of its distribution (e.g. any almost sure property) which is proved in such an extension (possibly using the random horofunction) will be valid in the original probability space as well.

In previous results of Gouëzel, Karlsson, and Ledrappier (see \cite{karlsson-ledrappier} and \cite{karlsson-goezel}) a random horofunction capturing the rate of escape of a distance stationary sequence is constructed without imposing any condition on the underlying probability space.  However, in their results the sequence of increments of such a horofunction is not guaranteed to be stationary.

Our result is very much related to the decomposition of stationary subadditive processes into the sum of a stationary additive process and a stationary purely subadditive one.  This result was first proved by Kingman to obtain the subadditive ergodic theorem (see \cite{kingman1968}, \cite{kingman1973}, and \cite{deljunco}).  In essence we modify the proof of the decomposition theorem via Komlos' theorem due to Burkholder (see the discussion by Burkholder in \cite{kingman1973}).  In this context the extension of the base probability space seems to be required to interpret the stationary additive process in the decomposition as the sequence of increments of a random horofunction.

The existence of a horofunction capturing the rate of escape of a sequence has several implications.  Notably, in spaces of negative curvature it implies the existence of a geodesic tracking the sequence.  As shown by Kaimanovich \cite{kaimanovich87} this is sufficient to obtain Oseledet's theorem.  A wealth of other applications are exhibited in previously cited works (also see \cite{karlsson-margulis}).

The additional fact that the increments of the horofunction are stationary, allows one to obtain explicit estimates on the speed in certain situations, and also to recapture the original formula due to Furstenberg of the Lyapunov exponent of a product of $2\times 2$ i.i.d.\ matrices in terms of a stationary probability measure on projective space (see \cite[Chapter 2, Theorem 3.6]{bougerol-lacroix}).   We will briefly discuss these types of applications.

\section{Preliminaries}

\subsection{Distance stationary sequences}

In what follows $(M,d)$ denotes a complete separable metric space and $o \in M$ a base point which is fixed from now on.

A random sequence $(x_n)_{n \in \Z}$ of points in $M$ is said to be \emph{distance stationary} if the distribution of $(d(x_m,x_n))_{m,n \in \Z}$ coincides with that of $(d(x_{m+1},x_{n+1}))_{m,n \in \Z}$.

\subsection{Horofunctions}

To each point $x \in M$ we associate a horofunction $\xi_x: M \to \R$ defined by
\[\xi_x(y) = d(x,o) - d(x,y).\]
The horofunction compactification of $M$ is the space $\widehat{M}$ obtained as the closure of the functions of the form $\xi_x$ in the topology of uniform convergence on compact sets.

Compactness of $\widehat{M}$ follows from the Arsel\`a-Ascoli theorem and the fact that all functions $\xi_x$ are $1$-Lipschitz.  A horofunction on $M$ is an element of $\widehat{M}$.

Horofunctions which are not of the form $\xi_x$ will be called \emph{boundary horofunctions} and the set of boundary horofunctions is the horofunction boundary of $M$, which might sometimes be written $\widehat{M} \setminus M$ abusing notation slightly.

\subsection{Speed or linear drift}

If $(x_n)_{n \in \Z}$ is a distance stationary sequence in $M$ satisfying $\E{d(x_0,x_1)} < +\infty$, then by Kingman's subadditive ergodic theorem the random limit
\[\ell = \lim\limits_{n \to +\infty}\frac{1}{n}d(x_0,x_n)\]
exists almost surely and in $L^1$.

We call this limit the \emph{speed} or \emph{linear drift} of $(x_n)_{n \in \Z}$.

\subsection{Stationary sequences and Birkhoff limits}

Recall that a random sequence $(s_n)_{n \in \Z}$ is stationary if its distribution coincides with that of $(s_{n+1})_{n \in \Z}$.  

If $(s_n)_{n \in \Z}$ is stationary and $\E{|s_1|} < +\infty$, then by Birkhoff's ergodic theorem the limit 
\[\lim\limits_{n \to +\infty}\frac{1}{n}\sum\limits_{i = 0}^{n-1}s_i\]
exists almost surely and in $L^1$.  We call this limit the Birkhoff limit of the sequence.

\subsection{Spaces of probability measures and representations}

Given a Polish space $X$ we use $\mathcal{P}(X)$ to denote the space of Borel probability measures on $X$ endowed with the topology of weak convergence (i.e. a sequence converges if the integral of each continuous bounded function from $X$ to $\R$ does).   This space is also Polish and is compact if $X$ is compact.

We will use the following result due to Blackwell and Dubins (see \cite{blackwell-dubins1983}):
\begin{theorem}[Continuous representation of probability measures]\label{blackwell-dubinstheorem}
For any Polish space $X$ there exists a function $F:\mathcal{P}(X)\times [0,1] \to X$ such that if $u$ is a uniform random variable on $[0,1]$ the following holds:
\begin{enumerate}
 \item For each $\mu \in \mathcal{P}(X)$ The random variable $F(\mu,u)$ has distribution $\mu$.
 \item If $\mu_n \to \mu$ then $F(\mu_n,u) \to F(\mu,u)$ almost surely.
\end{enumerate}
\end{theorem}

We call a function $F$ satisfying the properties in the above theorem a continuous representation of $\mathcal{P}(X)$.

\subsection{Application of Komlos' theorem to random probabilities}

Recall that a sequence $(a_n)_{n \ge 1}$ is Cesaro convergent if the limit $\lim\limits_{n \to +\infty}\frac{1}{n}\sum\limits_{k = 1}^n a_k$ exists.  We restate the main result \cite{komlos1967}.

\begin{theorem}[Komlos' theorem]
Let $(X_n)_{n \ge 1}$ be a sequence of random variables with $\sup\limits_n\E{|X_n|} < +\infty$.  Then there exists a subsequence $(Y_n)_{n \ge 1}$ of $(X_n)_{n \ge 1}$ which Cesaro converges almost surely to a random variable $Y$ with finite expectation, and furthermore any subsequence of $(Y_n)_{n \ge 1}$ has the same property.
\end{theorem}

We will need the following corollary of Komlos' theorem.
\begin{corollary}\label{komloscorollary}
Let $(\mu_n)_{n \ge 0}$ be a sequence of random probabilities on a compact metric space $(X,d)$.  There exists a subsequence $(\mu_{n_k})_{k \ge 1}$ which Cesaro converges almost surely to a random probability $\mu$ on $X$.
\end{corollary}
\begin{proof}
In this proof we use the notation $\nu(f) = \int\limits_{X} f(x)d \nu(x)$.

Let $(f_n)_{n \ge 1}$ be a dense sequence in the space of continuous functions from $X$ to $\R$ (with respect to the topology of uniform convergence).

 Applying Komlos' theorem to $\left(\mu_n(f_1)\right)_{n \ge 1}$ one obtains a subsequence $n_{1,k} \to +\infty$ such that $\mu_{n_{1,k}}(f)$ Cesaro converges almost surely and any further subsequence has the same property.
 
 For $i = 1,2,3,4,\ldots$, inductively applying Komlos' theorem to $(\mu_{n_{i,k}}(f_{i+1}))_{k \ge 1}$ we obtain a subsequence $(n_{i+1,k})_{k \ge 1}$ of $(n_{i,k})_{k \ge 1}$ such that $\mu_{n_{i+1,k}}(f_{i+1})$ Cesaro converges almost surely and any further subsequence has the same property.
 
 Setting $n_k = n_{k,k}$ one obtains that $(\mu_{n_k})_{n \ge 1}$ Cesaro converges to a random probability $\mu$ almost surely.
\end{proof}

\section{Furstenberg type formula for speed}

\subsection{Statement and proof}

\begin{theorem}[Furstenberg type formula for distance stationary sequences]\label{furstenbergtheorem}
Let $(x_n)_{n \in \Z}$ be a distance stationary sequence in a complete separable metric space $(M,d)$ satisfying $\E{d(x_0,x_n)} < +\infty$ and $\ell$ be its linear drift.

Suppose there exists a random variable $u$ which is uniformly distributed on $[0,1]$ and independent from $(x_n)_{n \in \Z}$.  Then the following holds:
\begin{enumerate}
 \item The sequence of random probability measures on $\widehat{M}$ defined by $\mu_n = \frac{1}{n}\sum\limits_{i = 1}^n \delta_{\xi_{x_{-i}}}$ has a subsequence which is almost surely Cesaro convergent to a random probability $\mu$.
 \item There exists a random horofunction $\xi$ which is measurable with respect to $\sigma(u,\mu)$ and whose conditional distribution given $(x_n)_{n \in \Z}$ is $\mu$.
 \item The sequence of increments $(\xi(x_n)-\xi(x_{n+1}))_{n \in \Z}$ is stationary and its Birkhoff limit equals $\ell$ almost surely.  In particular, $\E{\ell} = \E{\xi(x_0)-\xi(x_1)}$.
\end{enumerate}
\end{theorem}
\begin{proof}
The fact that $(\mu_n)_{n \ge 1}$ has an almost surely Cesaro convergent subsequence follows directly from the version of Komlos' theorem for random probabilities given above (see Corollary \ref{komloscorollary}).  Let $(\mu_{n_j})_{j \ge 1}$ be such a subsequence and $\mu$ be its almost sure Cesaro limit.

Let $F:\mathcal{P}(\widehat{M})\times [0,1] \to \widehat{M}$ be continuous representation of $\mathcal{P}(\widehat{M})$, as given by Theorem \ref{blackwell-dubinstheorem} and define $\xi = F(u,\mu)$.  Clearly $\xi$ is $\sigma(u,\mu)$-measurable and its conditional distribution given $(x_n)_{n \in \Z}$ is $\mu$.

We will now show that $\E{\xi(x_0)-\xi(x_1)} = \ell$.  

For this purpose let $\xi_k = F(u,\frac{1}{k}\sum\limits_{j = 1}^k\mu_{n_j})$ and notice that $\xi_k \to \xi$ almost surely when $k \to +\infty$.  Because horofunction are $1$-Lipschitz one has $|\xi_k(x_0)-\xi_k(x_1)| \le d(x_0,x_1)$.   Since $\E{d(x_0,x_1)} < +\infty$ this implies that the sequence is uniformly integrable and one obtains $\E{\xi(x_0)-\xi(x_1)} = \lim\limits_{k \to +\infty}\E{\xi_k(x_0) - \xi_k(x_1)}$.

For the sequence on the right hand side using distance stationarity one obtains

\begin{align*}
\E{\xi_k(x_0) - \xi_k(x_1)} &= \frac{1}{k}\sum\limits_{j = 1}^k\E{\frac{1}{n_j}\sum\limits_{i = 1}^{n_j} \xi_{x_{-i}}(x_0) - \xi_{x_{-i}}(x_1)}
\\ &= \frac{1}{k}\sum\limits_{j = 1}^k\E{\frac{1}{n_j}\sum\limits_{i = 1}^{n_j} -d(x_{-i},x_0) + d(x_{-i},x_1)}
\\ &= \frac{1}{k}\sum\limits_{j = 1}^k\E{\frac{1}{n_j}\sum\limits_{i = 1}^{n_j} -d(x_{0},x_i) + d(x_0,x_{i+1})}
\\ &= \frac{1}{k}\sum\limits_{j = 1}^k\E{\frac{-d(x_0,x_1) + d(x_0,x_{n_j+1})}{n_j}}.
\end{align*}

Taking the limit when $k \to +\infty$ above it follows that $\E{\xi(x_0)-\xi(x_1)} = \E{\ell}$ as claimed.

We will now prove that $\left(\xi(x_n) - \xi(x_{n+1})\right)_{n\in \Z}$ is stationary.  

Suppose $F:\R^\Z \to \R$ is continuous and bounded and notice that by distance stationarity one has
\begin{equation*}
\begin{split}
\mathbb{E}(F((\xi_k(x_{n+1}) &- \xi_k(x_{n+2}))_{n \in \Z}))\\ &= \frac{1}{k}\sum\limits_{j = 1}^k\frac{1}{n_j}\sum\limits_{i = 1}^{n_j}\E{ F\left(\left(\xi_{x_{-i}}(x_{n+1}) - \xi_{x_{-i}}(x_{n+2})\right)_{n \in \Z}\right)}
\\ &= \frac{1}{k}\sum\limits_{j = 1}^k\frac{1}{n_j}\sum\limits_{i = 1}^{n_j}\E{ F\left(\left(\xi_{x_{-(i-1)}}(x_{n}) - \xi_{x_{-(i-1)}}(x_{n+1})\right)_{n \in \Z}\right)}
\\ &= \frac{1}{k}\sum\limits_{j = 1}^k\frac{1}{n_j}\sum\limits_{i = 0}^{n_j-1}\E{ F\left(\left(\xi_{x_{-i}}(x_{n}) - \xi_{x_{-i}}(x_{n+1})\right)_{n \in \Z}\right)}
\\ &= C_k \max |F| + \E{F\left(\left(\xi_k(x_{n}) - \xi_k(x_{n+1})\right)_{n \in \Z}\right)}
\end{split}
\end{equation*}
where $|C_k| \le \frac{1}{k}\sum\limits_{j = 1}^k \frac{2}{n_j} \to 0$ when $k \to +\infty$.

From this the stationarity of the increments of $\xi$ along the sequence $(x_n)_{n \in \Z}$, follows directly taking limit when $k \to +\infty$.

By Birkhoff's theorem the Birkhoff averages of the increments of $\xi$ along the sequence exist almost surely and in $L^1$.  Additionally, because horofunctions are $1$-Lipschitz, one has
\[\lim\limits_{n \to +\infty} \frac{1}{n}\sum\limits_{k = 0}^{n-1} \xi(x_k) - \xi(x_{k+1}) \le \ell\]
almost surely.  But the expectation of the left hand side in the above inequality is $\E{\xi(x_0)-\xi(x_1)} = \E{\ell}$.  Hence both sides coincide almost surely.  This concludes the proof.
\end{proof}

\subsection{Boundary horofuncions}

The question of whether the random horofunction $\xi$ given by Theorem \ref{furstenbergtheorem} is almost surely on the horofunction boundary of $M$ sometimes arises.

A trivial example where this is not the case is obtained by letting $(x_n)_{n \in \Z}$ be an i.i.d. sequence of uniformly distributed random variables on $[0,1]$.  In this case the horofunction $\xi$ given by Theorem \ref{furstenbergtheorem} will be uniformly distributed on $[0,1]$ and independent from the sequence.

In the previous example the linear drift $\ell$ was $0$ almost surely.  It is not difficult to show that if $\ell > 0$ almost surely then $\xi$ must be a boundary horofunction almost surely.

However, in many examples Theorem \ref{furstenbergtheorem} can be used to decide whether or not $\ell$ is positive.  Hence it is useful to have a criteria for establishing that $\xi$ is almost surely on the horofunction boundary without knowledge of $\ell$.   The following proposition is such a result.

\begin{proposition}\label{corollaryboundary}
Assume $(x_n)_{n \in \Z}$ is a distance stationary sequence with $\E{d(x_0,x_1)} < +\infty$ and $u$ is a random variable which is uniformly distributed on $[0,1]$ and independent from $(x_n)_{n \in \Z}$.   If $\Pof{x_n \in K} \to 0$ when $n \to -\infty$ for all bounded sets $K$, then the random horofunction $\xi$ given by Theorem \ref{furstenbergtheorem} is almost surely on the horofunction boundary.
\end{proposition}
\begin{proof}
We will use the notation from Theorem \ref{furstenbergtheorem}.  Let $\nu_n$ denote the sequence of averages of the subsequence of the probabilities $(\mu_n)$ which Cesaro converges to $\mu$ almost surely.

Given a bounded set $K$ pick a bounded open set $U$ containing the closure of $K$.  From the hypothesis it follows that $\E{\mu_n(U)} \to 0$ when $n \to +\infty$.   Therefore $\E{\nu_n(U)} \to 0$ as well. 

Because $\nu_n \to \mu$ one has $\mu(U) \le \liminf\limits_n \nu_n(U)$ almost surely.  Combining this with Fatou's lemma one obtains
\[\Pof{\xi \in K} = \E{\mu(K)} \le \E{\mu(U)} \le \E{\liminf\limits_n\nu_n(U)} \le \liminf\limits_n \E{\nu_n(U)} = 0,\]
and hence $\xi \notin K$ almost surely.
\end{proof}

\section{Applications}\label{applicationssection}

\subsection{Right-angled hyperbolic random walk}\label{rightangledsection}

To illustrate Theorem \ref{furstenbergtheorem} consider a \emph{right angled random walk} on the hyperbolic plane (see also \cite{gruet}).   That is, starting with a unit tangent vector $(o,v)$ consider the Markov process where at each step one rotates the vector a random multiple of 90º (each value $-90,0,90,180$ having the same probability) and then advances in direction of the geodesic a distance $r > 0$.   Suppose the sequence of base points thus obtained is $x_0=o,x_1,x_2,\ldots$.   One can extend this to a bi-infinite distance stationary sequence by letting $x_0,x_{-1},x_{-2},\ldots$ be an independent random walk constructed in the same way.

\begin{figure}
\centering
\includegraphics[scale=0.45]{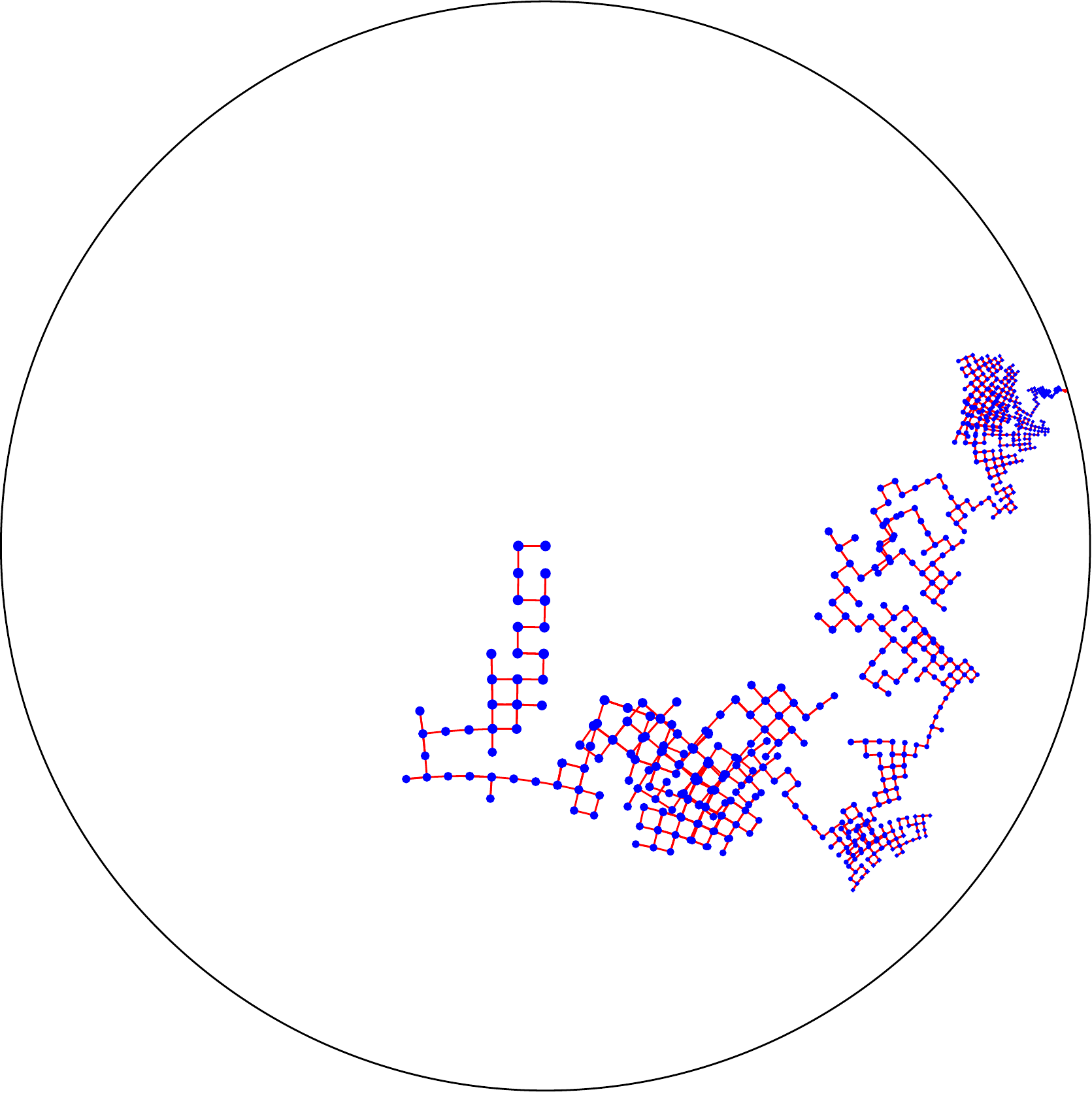}
\caption{\label{rightangledfigure}A realization of a right angled random walk in the Poincaré disk for step size $r = 0.1$.  Hyperbolic segments have been added between consecutive points of the walk, notice that there are no squares in the hyperbolic plane, so any apparent square in the figure does not in fact close up.}
\end{figure}

For $n = 5,6,7,\ldots$ let $r_n$ be the side of the regular $n$-gon with interior right angles in the hyperbolic plane.  Notice that if $r = r_n$ , then the walk $\left(x_n\right)_{n \in \Z}$ remains on the vertices of a tessellation by regular $n$-gons with 4 meeting at each vertex.

Setting $r_\infty = \lim\limits_{n \to +\infty}r_n$ one may show (via a ping-pong argument on the boundary) that if $r \ge r_{\infty}$ the random walk remains on the vertices of an embedded regular tree of degree 4.

For all other values of $r$ (smaller than $r_\infty$ but not one of the $r_n$) it seems clear that the set of points attainable by the random walk is dense in the hyperbolic plane (though a short argument is not known to the authors).  For example, for $r$ small enough this follows by Margulis' lemma, while for a dense set of values of $r$ the elliptic element relating the initial unit vector to the one obtained by advancing $r$ and then rotating 90º is an irrational rotation.

The speed $\ell_r = \lim\limits_{n \to +\infty}\frac{d(x_0,x_n)}{n}$ exists for all $r$ since $d(x_0,x_1) = r$ almost surely. 

We will now sketch how Theorem \ref{furstenbergtheorem} may be used to show that $\ell_r > 0$ almost surely for all $r > 0$.

The first step is to establish that $\P(d(x_0,x_n) < C) \to 0$ when $n \to +\infty$ for all $C > 0$.   This can be shown by first observing that there exists a sequence $(g_n)_{n \in \Z}$ of i.i.d.\ isometries of the hyperbolic plane such that $x_n = g_1\circ \cdots \circ g_n (x_0)$ for all $n \ge 0$.   Since the distribution of $g_1$ is not supported on a compact subgroup of the isometry group the distribution of $g_1\circ \cdots \circ g_n$ goes to zero on any compact set as follows for example from \cite[Theorem 8]{MR0423532}.

Furthermore, since $\ell$ is tail measurable with respect to the sequence $(g_n)_{n \in \Z}$, it follows from Kolmogorov's zero-one law that $\ell$ is almost surely constant.

By Theorem \ref{furstenbergtheorem} and Proposition \ref{corollaryboundary}, one obtains that there exists a random boundary horofunction $\xi$ which is independent from $x_1$ and such that
\[\ell = \E{\ell} = -\E{\xi(x_1)}.\]

Even though the distribution of $\xi$ is unknown (e.g. a priori it need not be uniform on the boundary circle, though it must be invariant under 90 degree rotation), one may use the existence of $\xi$ to show that $\ell > 0$.

For this purpose notice that $x_1$ takes four values, say $a,b,c,d$ with equal probability $1/4$.  Conditioning on $\xi$ (using the independence of $\xi$ and $x_1$) one obtains:
\[\ell = -\E{\frac{\xi(a)+\xi(b)+\xi(c)+\xi(d)}{4}}.\]

Finally the result follows because $\xi(a)+\xi(b)+\xi(c)+\xi(d) < 0$ for all boundary horofunctions $\xi$.  To see this one may calculate in a concrete model.  For example in the Poincaré disk if the starting point is $0$ and the initial unit tangent vector points towards the positive real axis, one may take $a = x, b = ix, c = -x, d = -ix$ where $x = \tanh(r/2)$.  The boundary horofunctions are of the form $\xi(z) = \log\left(\frac{1-|z|^2}{|z-e^{i\theta}|^2}\right)$ for $\theta \in [0,2\pi]$ (see for example \cite[Section 8.24]{bridson-haefliger}).

Hence one obtains
\begin{align*}
\xi(a)+\xi(b)+\xi(c)+\xi(d) &= \log\left(\frac{(1-x^2)^4}{|e^{4i\theta}-x^4|^2}\right)
\\& \le \log\left(\frac{(1-x^2)^2}{(1+x^2)^2}\right) = -\int_0^r \frac{\tanh(s)}{2} ds < 0.
\end{align*}

This yields an explicit lower bound on the speed for each $r > 0$.  The lower bound is equivalent to $r^2/4$ when $r\to 0$ and to $r/2$ when $r \to +\infty$.  For large $r$ the bound is close to optimal because the random walk on a regular tree of degree 4 has linear drift $1/2$ with respect to the graph distance.

This examples illustrates the method which we will use later on to establish positive speed for hyperbolic Poisson-Delaunay random walks.  A major difference is that in the Poisson-Delaunay case one may no longer guarantee that $\xi$ is independent from $x_1$.

\subsection{Lyapunov exponents of $2\times 2$ i.i.d. matrix products}

Suppose that $(A_n)_{n \in \Z}$ is an i.i.d. sequence of matrices in $\text{SL}(2,\R)$ with the additional property that $\E{\log(|A_1|)} < +\infty$ where $|A|$ denotes the operator norm of the matrix $A$.

The largest Lyapunov exponent of the sequence $(A_n)$ is defined by
\[\chi = \lim\limits_{n \to +\infty}\frac{1}{n}\log\left(|A_n\cdots A_1|\right),\]
and is almost surely constant since it is a tail function of the sequence.

Notice that if one writes $A_n \cdots A_1 = O_n P_n$ where $O_n$ is orthogonal and $P_n$ is symmetric with positive eigenvalues, one obtains
\[\chi = \lim\limits_{n \to +\infty}\frac{1}{n}\log\left(|P_n|\right).\]

This implies that $\xi$ depends only on the sequence of projections of $A_n\cdots A_1$ to the left quotient $M = \text{SO}(2)\backslash\text{SL}(2,\R)$.  Let $[A]$ denote the equivalence class of a matrix $A \in \text{SL}(2,\R)$ in the quotient above.

The quotient space admits a (unique up to homotethy) Riemannian metric for which the transformations $[A] \mapsto [AB]$ are isometries for all $B \in \text{SL}(2,\R)$.  One may choose such a metric so that the distance $d([\text{Id}],[A]) = \sqrt{\log(\sigma_1)^2 + \log(\sigma_2)^2}$ where $\sigma_1,\sigma_2$ are the singular values of $A$ and $\text{Id}$ denotes the identity matrix.  In particular, since $\sigma_1 = |A|$ and $A$ has determinant $1$, one obtains $d(\text{Id},[A]) = \sqrt{2}\log(|A|)$.

With the Riemannian metric under consideration the sequence 
\[\ldots, x_{-2} = [A_{-1}^{-1}A_0^{-1}], x_{-1} = [A_0^{-1}], x_0 = [\text{Id}], x_1 =[A_1], x_2 = [A_2A_1],\ldots\]
is distance stationary and satisfies $\E{d(x_0,x_1)} < +\infty$.  Furthermore, its rate of escape is $\ell = \sqrt{2}\chi$.

The boundary horofunctions on $M$ are of the form $\xi([A]) = -\sqrt{2}\log(|Av|)$ for some $|v| = 1$ (see for example \cite{MR1794514}).

If $A_1$ is not contained almost surely in a compact subgroup of $\text{SL}(2,\R)$ then one may use \cite[Theorem 8]{MR0423532} to show that $\P(x_n \in K) \to 0$ when $n \to +\infty$ for all compact sets $K \subset M$.

Hence, Theorem \ref{furstenbergtheorem} and Proposition \ref{corollaryboundary} imply the existence of a random unit vector $v \in \R^2$ which is independent from $A_1$ and such that
\[\chi = \E{\log\left(|A_1v|\right)}.\]

In particular, letting $\mu$ be the distribution of $A_1$, there is a probability $\nu$ on the unit circle $S^1 \subset \R^2$ such that 
\[\chi = \int\limits_{\text{SL}(2,\R)} \int\limits_{S^1} \log\left(|Av|\right) \mathrm{d}\nu(v) \mathrm{d}\mu(A).\]

This is typically called Furstenberg's formula for the largest Lyapunov exponent (see \cite[Theorem 3.6]{bougerol-lacroix}).   It follows from Theorem \ref{furstenbergtheorem} that $\nu$ is $\mu$-stationary (where the action of $\text{SL}(2,\R)$ on $S^1$ is by transformations of the form $v \mapsto Av/|Av|$).  This may be used as a starting point to establish a criteria for an i.i.d.\ random matrix product to have a positive Lyapunov exponent.

Also, in some cases, formulas of this type can be used to give explicit estimates for the largest Lyapunov exponent in a family of random matrix products depending on some parameter (see for example \cite{MR3623241}, and \cite{MR715727}).

The reasoning above may be carried out in $\text{SL}(n,\R)$ for larger $n$.  What results is a formula for the sum of squares of the Lyapunov exponents of the random i.i.d.\ product of matrices.  As above, the distribution of the random boundary horofunction is unknown (in larger dimension horofunctions are determined by a choice of a flag and a sequence of weights adding up to zero).

\part{Distance stationarity of Poisson-Delaunay random walks\label{stationaritypart}}

Throughout this part of the article $M$ will be a Riemannian symmetric space, $o \in M$ a fixed base point, and $P$ a homogeneous Poisson point process in $M$ with constant intensity $\lambda$ (i.e. $\lambda$ points per unit volume).

We say two distinct points $x,y$ in a discrete subset $X$ of $M$ are Delaunay neighbors if there exits an open ball in $M$ which is disjoint from $X$ and contains $x$ and $y$ on its boundary.  This gives the set $X$ a graph structure by adding an undirected edge between each pair of Delaunay neighbors.  We call this graph the Delaunay graph associated to $X$.

The Voronoi cell of a point $x$ in a discrete set $X$ is the set
\[V_x = \lbrace y \in M: d(x,y) = d(X,y)\rbrace.\]
An alternative definition of the Delaunay graph is obtained by noticing that two distinct point $x,y \in X$ are Delaunay neighbors if and only if $V_x \cap V_y \neq \emptyset$.

In what follows we will consider the Delaunay graph of the set $P_o = P \cup \lbrace o\rbrace$ rooted at $o$.  This is a Poisson-Delaunay random graph (see \cite{benjamini-paquette-pfeffer}). See Figure \ref{poissondelaunayfigures} for some examples in the hyperbolic plane.

\begin{figure}
\centering
\begin{subfigure}{.5\textwidth}
  \centering
  \includegraphics[width=.9\linewidth]{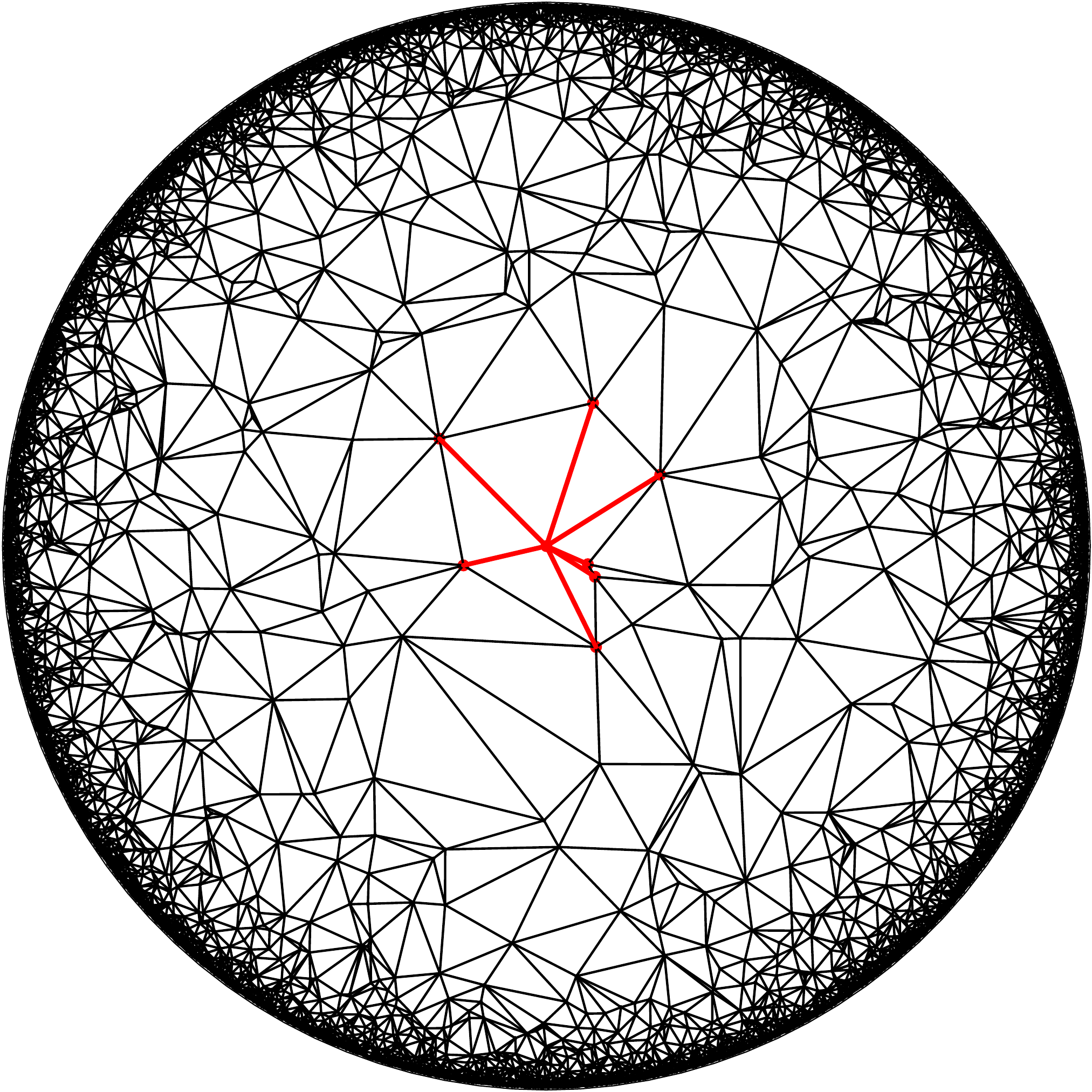}
  \caption{$\lambda=10$}
\end{subfigure}%
\begin{subfigure}{.5\textwidth}
  \centering
  \includegraphics[width=.9\linewidth]{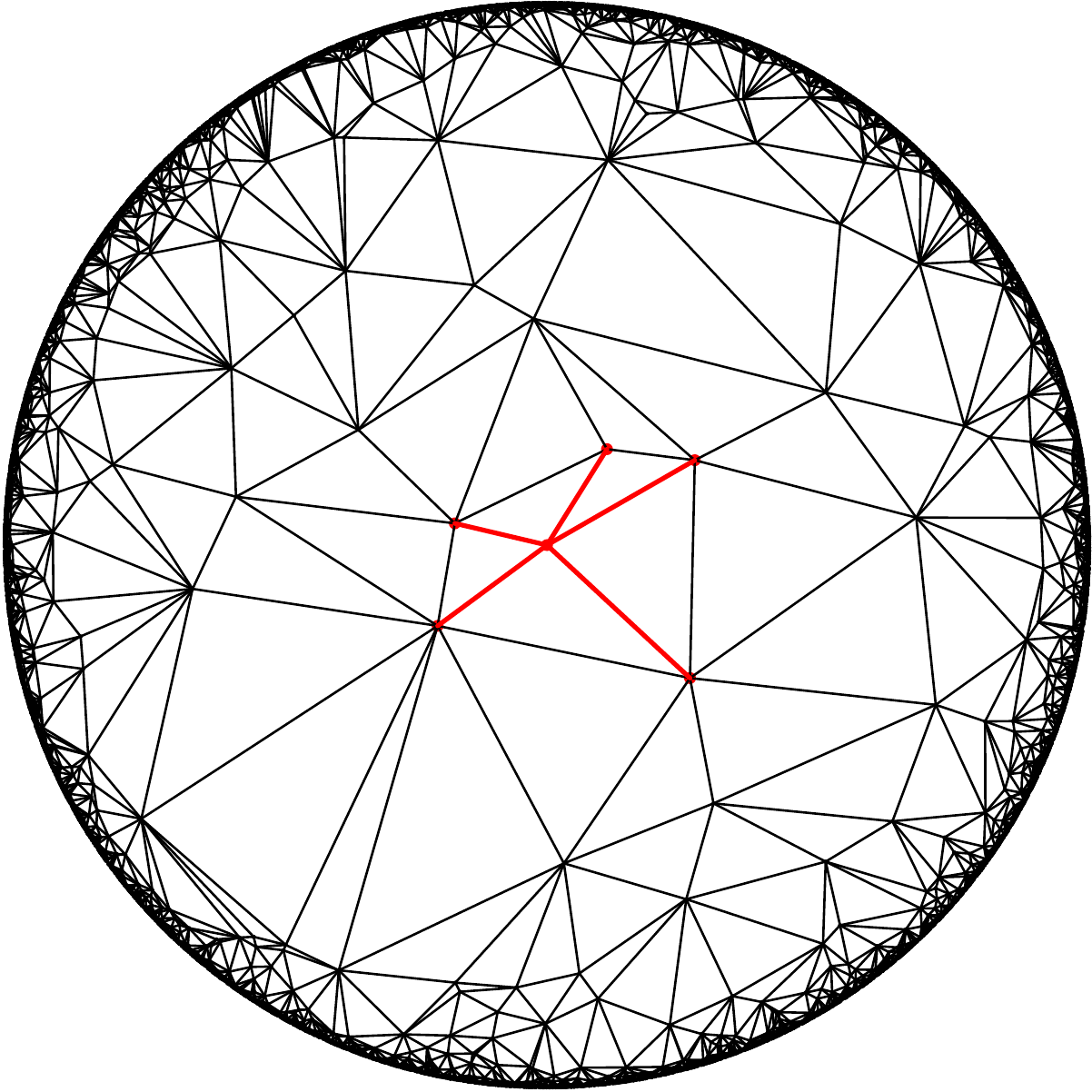}
  \caption{$\lambda=1$}
 \end{subfigure}%

\begin{subfigure}{.5\textwidth}
  \centering
  \includegraphics[width=.9\linewidth]{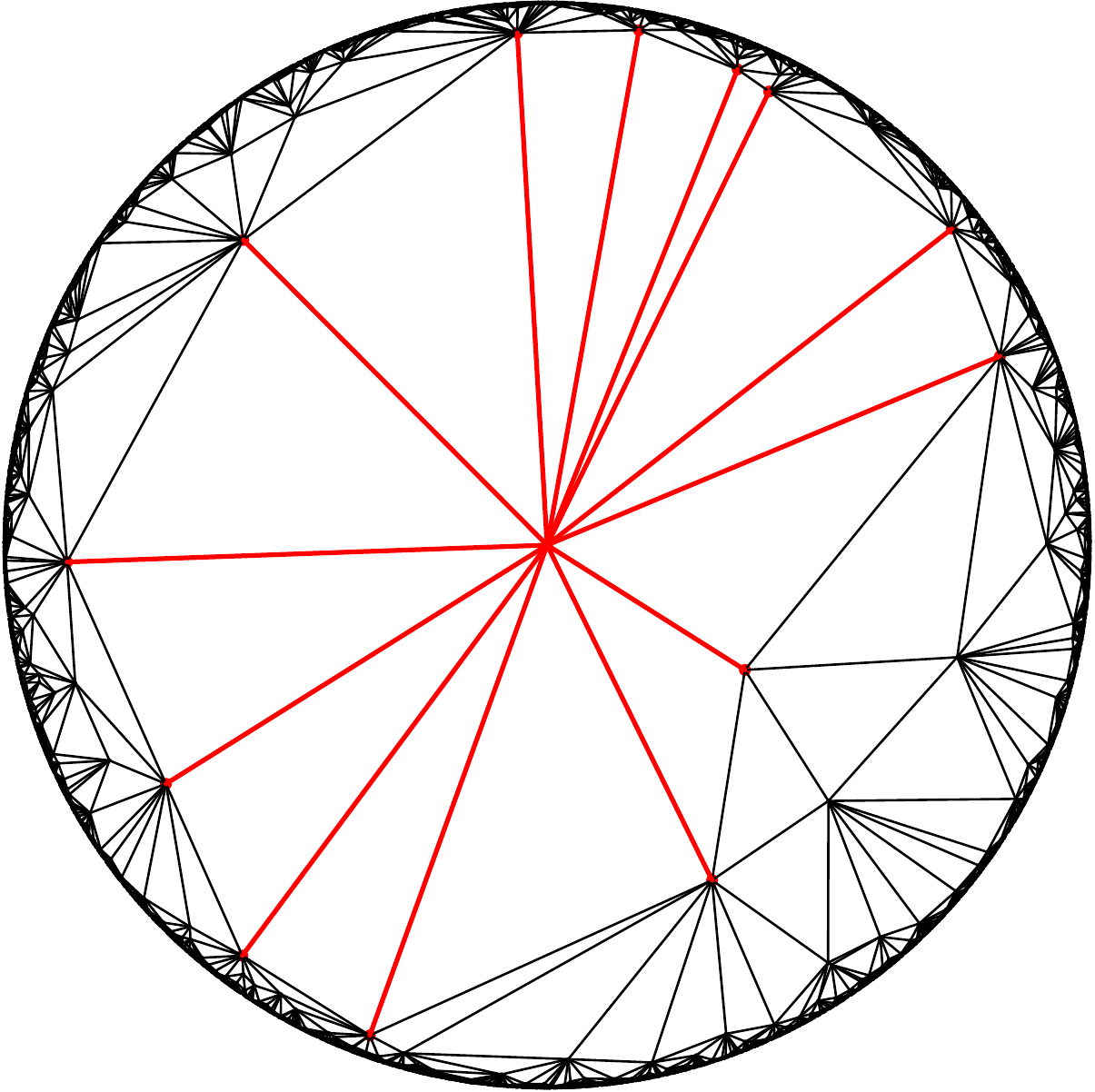}
  \caption{$\lambda=0.25$}
\end{subfigure}%
\caption{\label{poissondelaunayfigures} Illustration of three realizations for different values of $\lambda$ of Poisson-Delaunay graphs in the hyperbolic plane. The neighbors of $o$ are indicated in red.}
\end{figure}

Such graphs are known to be unimodular, and stationary under a suitable bias.  We will prove a slight generalization of these facts where we take into account the embedding of the graph in the ambient space $M$.

Using this we will construct a distance stationary sequence related to the simple random walk on the Poisson-Delaunay random graph.

\section{Degree biased probability and distance stationarity}

\subsection{Unimodularity}

For each $x \in M$ we denote by $g_x$ a central symmetry exchanging $o$ and $x$ chosen measurably as a function of $x$ (if $M$ is of non-compact type $g_x$ is unique for all $x$).

The space of discrete subsets of $M$ will be denoted by $\discrete{M}$.  We consider the natural topology on $\discrete{M}$ where each discrete set is identified with a counting measure and convergence is equivalent to convergence of the integrals of all continuous functions with compact support.  With this topology $\discrete{M}$ is separable and completely metrizable (i.e. a Polish space).

We assume all random objects in this section are defined on the same fixed probability space which we denote by $(\Omega,\F,\P)$.

In what follows we will use Slivyak's formula (sometimes called Mecke's formula) which allows one to calculate the expected values of the sum of a function over all points in $P$ (where the function may depend on $P$) as integrals over $M$.  We refer to \cite[Section 4.4]{stochasticgeometry} for a proof of this result (the context there is point processes in $\R^n$ but the same arguments go through on a Riemannian homogeneous space).

\begin{proposition}[Unimodularity]\label{unimodularity}
For every Borel function $F:\discrete{M}\times M \to [0,+\infty)$ one has
\[\E{\sum\limits_{x \in P}F\left(P_o, x\right)} = \E{\sum\limits_{x \in P}F\left(g_{x}P_o, x\right)}.\]
\end{proposition}
\begin{proof}
By Slivnyak's formula one has
\[\E{\sum\limits_{x \in P}F\left(P_o, x\right)} = \int_M \E{F(P \cup \lbrace o,y\rbrace , y)} \lambda dy,\]
where integration is with respect to the volume measure on $M$.

For each fixed $y$, one has that $g_y(P \cup \lbrace o,y\rbrace) = g_y P \cup \lbrace o,y\rbrace$ which has the same distribution as $P \cup \lbrace o,y\rbrace$.  Hence, the right-hand side of the equation above equals
\[\int_M \E{F(g_y P \cup \lbrace o, y\rbrace , y)} \lambda dy=\E{\sum\limits_{x \in P}F\left(g_{x}P_o, x\right)}.\]
In the last inequality we used again Slivnyak's formula.
\end{proof}

\subsection{Reversibility under the degree biased probability}

Let $\deg(o)$ denote the number of Delaunay neighbors of $o$ in the Delaunay graph of $P_o$.   It follows from \cite[Theorem 3.3]{paquette} that $\E{\deg(o)} < +\infty$.

We define the degree biased probability $\Pdg$ on $(\Omega,\F,\P)$ by
\[\frac{d \Pdg}{d \P} = \frac{\deg(o)}{\E{\deg(o)}}.\]

Expectation with respect to the degree biased probability is denoted by $\Edg{\cdot}$.

Let $x_1$ be a uniform random Delaunay neighbor of $o$ in $P_o$; i.e.\ given $P_o$, $x_1$ has uniform distribution among the neighbors of $o$.

\begin{proposition}[Reversibility]\label{reversibility}
Under the degree biased probability the distribution of $\left(P_o ,x_1\right)$ is the same as that of $\left(g_{x_1}P_o,  x_1\right)$.
\end{proposition}
\begin{proof}
Given any Borel function $F:\discrete{M}\times M \to [0,+\infty)$ one has
\begin{align*}
\Edg{F(P_o, x_1)}\E{\deg(o)} &= \E{\sum\limits_{x \in P}F(P_o, x)\mathds{1}_{[x \sim o]}}
\\ &= \E{\sum\limits_{x \in P}F(g_x P_o, x)\mathds{1}_{[x \sim o]}}
\\ &= \Edg{F(g_{x_1} P_o, x_1)}\E{\deg(o)},
\end{align*}
where the second equality follows from Proposition \ref{unimodularity}.  Here $x \sim o$ means that $x$ is a Delaunay neighbor of $o$ in the discrete set under consideration.  Notice that $x \sim o$ in $P_o$ if and only if $x \sim o$ in $g_xP_o$. 

Since this is valid for all choices of $F$, the distributions must coincide as claimed.
\end{proof}

\subsection{Local finiteness}

We say a discrete set $X \subset M$ intersects all horoballs in $M$ if for every sequence of balls $D_1,D_2,\ldots$ such that the radius of $D_n$ goes to infinity with $n$, and all $D_n$ intersect some fixed compact set $K \subset M$, one has $\bigcup\limits_n D_n \cap X \neq \emptyset$.

\begin{lemma}[Poisson processes intersect all horoballs]\label{noemptyhoroballs}
Almost surely, $P$ intersects all horoballs.
\end{lemma}
\begin{proof}
If $M$ compact the statement is trivial.  We assume from now on that $M$ is non-compact.

Consider for each $n = 1,2,\ldots$ a maximal $n/3$-separated subset $S_n$ of the the boundary $\partial B_n$ of the ball of radius $n$ centered at $o$, and let $A_n$ be the set of balls of radius $n/3$ centered at points in $S_n$.

Notice that $A_n$ is an open covering of $\partial B_n$.   Furthermore if $V(r)$ denotes the volume of any ball of radius $r$ in $M$ one has that the number of elements $N_n$ in $A_n$ is at most $V(4n/3)/V(n/6)$.

We claim that almost surely, for all $n$ large enough every ball in $A_n$ intersects $P$.

To see this we calculate
\[\Pof{P \cap B_{n/3}(x) = \emptyset\text{ for some }x \in S_n} \le  N_n e^{-\lambda V(n/3)} \le \frac{V(4n/3)}{V(n/6)}e^{-\lambda V(n/3)}.\]

Since $M$ is a non-compact symmetric space one has that either $V(r)$ is bounded between two polynomials of the same degree which is at least $1$ (if the only non-compact factor in the de Rham splitting of $M$ is Euclidean) or there exist positive constants $a < b$ such that $e^{ar} \le V(r) \le e^{br}$ for all $r$ large enough (if there is a symmetric space of non-compact type in the de Rham splitting of $M$).  In both cases the right hand term above is summable in $n$.  Hence, applying the Borel-Cantelli Lemma establishes the claim.

Suppose now that $c_n$ is a sequence of points in $M$ and $r_n$ an unbounded sequence of radii such that the open balls $D_n = B_{r_n}(c_n)$ of radius $r_n$ centered at $c_n$ satisfy for $d(o,D_n) \le C$ for some fixed positive constant $C$.

Let $R_n$ be the integer part of $d(o,c_n)$ and $x_n$ a point in $\partial B_{R_n}$ which minimizes the distance to $c_n$.  Choose $y_n \in S_n$ such that $d(x_n,y_n) < R_n/3$.

Notice that $d(y_n,c_n) \le d(y_n,x_n) + d(x_n,c_n) \le 1+R_n/3$.  On the other hand, picking a minimizing geodesic from $o$ to $c_n$, one has $r_n = d(o,c_n)-d(o,D_n) \ge d(o,c_n) - C \ge R_n - C$.  Hence for all $n$ large enough $B_{R_n/3}(y_n) \subset D_n$ and therefore almost surely there exists $n$ such that $D_n \cap P \neq \emptyset$.
\end{proof}

An important consequence of the above lemma is that almost surely every point in $P$ has a finite number of Delaunay neighbors.  Recall that the Voronoi cell of a point $x$ in a discrete set $X$ is the set of points $y$ satisfying $d(x,y) = d(X,y)$ (i.e. $y$ at least as close to $x$ as it is to any other point in $X$).

\begin{corollary}[Poisson-Delaunay graphs are locally finite]
Almost surely, the Poisson-Delaunay graph in a symmetric space is locally finite and all Voronoi cells are bounded.
\end{corollary}
\begin{proof}
The corollary follows from the claim that if a discrete set $X$ intersects all horoballs then all its Voronoi cells are bounded and every point in $X$ has a finite number of Delaunay neighbors.

To establish the claim first suppose that some point $x \in X$ has an infinite number of Delaunay neighbors.  Notice that for each neighbor $y$ of $X$ there exists an open ball with $x$ and $y$ on its boundary which is disjoint from $X$.  Since $X$ is discrete this gives a sequence of balls with unbounded radii with $x$ on their boundary and disjoint from $X$.  This would contradict the fact that $X$ intersects all horoballs.

On the other hand if the Voronoi cell of some point $x$ were unbounded one may take an unbounded sequence of points $y_n$ which are closer to $x$ than to any other point in $X$.   In this case the sequence of balls centered at the $y_n$ and with $x$ on their boundary would contradict the fact that $X$ intersects all horoballs.
\end{proof}

\subsection{Distance stationarity}

A Delaunay random walk on a discrete set $X \subset M$ is a simple random walk on its Delaunay graph.  Such a walk is well defined only if the Delaunay graph of $X$ is locally finite.

Let $\left(x_n\right)_{n \in \Z}$ be defined so that $\left(x_n\right)_{n  \ge 0}$ and $\left(x_{-n}\right)_{n \ge 0}$, conditioned on $P_o$, are two independent Delaunay random walks on $P_o$ starting at $o$. Let $y_n = x_{-n}$, then by definition $(P_o, \left(x_n\right)_{n \in \Z})$ and $(P_o, \left( y_n\right)_{n \in \Z})$ have the same distribution.   

We call a process $\left(x_n\right)_{n \in \Z}$ as defined in the previous paragraph a Delaunay random walk on $P_o$ starting at $o$, or a Poisson-Delaunay random walk.

\begin{figure}
\centering
\includegraphics{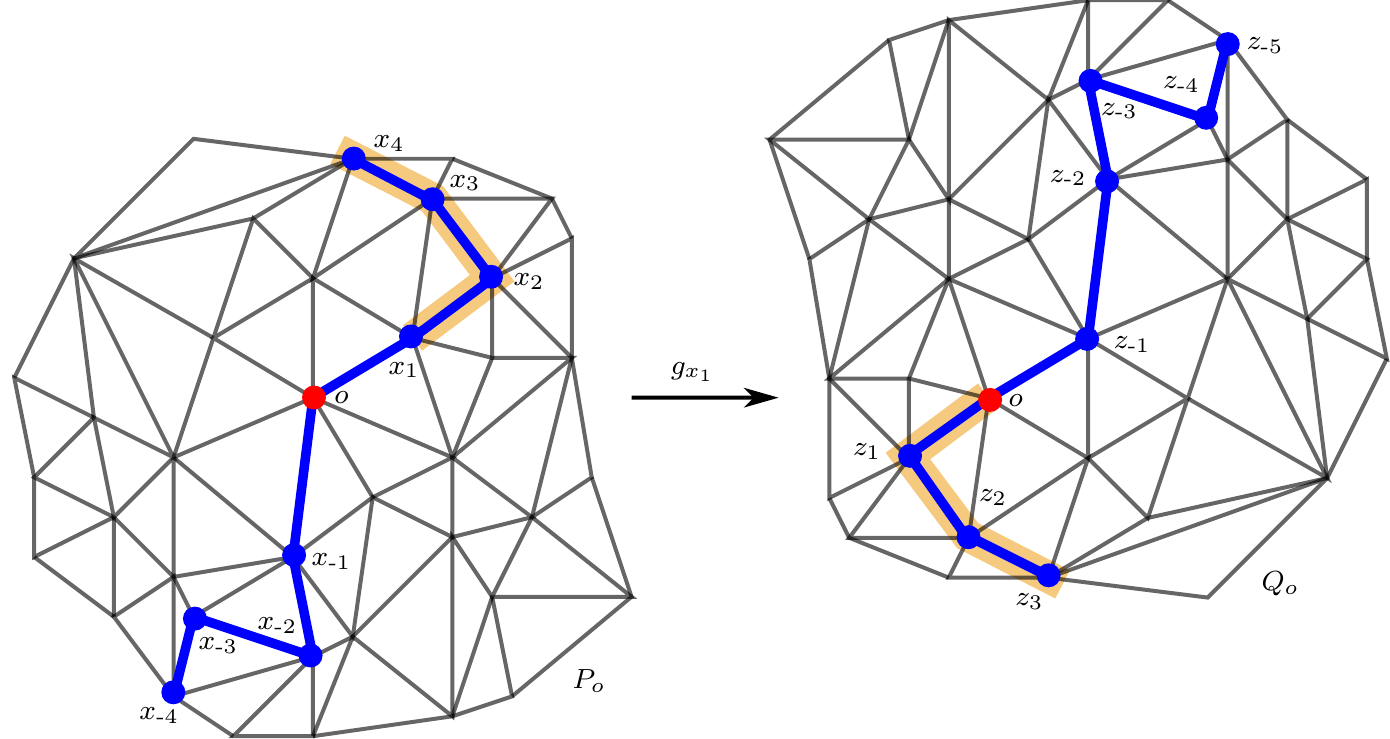}
\caption{This figure illustrates the proof of Theorem \ref{stationarity}. Conditioned on $P_o$ and $x_1$, the sequence $(x_n)_{n\geq 1}$ is a Poisson-Delaunay random walk starting at $x_1$. In particular, (conditioned on $Q_o$) $z_1$ is a uniformly distributed random neighbor of $o$ in $Q_o$ which is independent from $z_{-1}$.}
\end{figure}

\begin{theorem}[Distance stationarity]\label{stationarity}
Let $\left(x_n\right)_{n \in \Z}$ be a Delaunay random walk starting at $o$ on $P_o$ where $P$ is a constant intensity Poisson point process on a Riemannian symmetric space $M$ with base point $o$.

Then, under the degree biased probability, the distribution of $(P_o,\left(x_n\right)_{n \in \Z})$ is the same as that of $(g_{x_1}P_o,\left( g_{x_1}x_{n+1}\right)_{n \in \Z})$.  

In particular, the sequence $\left(x_n\right)_{n \in \Z}$ is distance stationary under the degree biased probability.
\end{theorem}
\begin{proof}
In all of what follows we will work under the probability $\Pdg$. Let $z_n = g_{x_1}x_{n+1}$ and $Q_o = g_{x_1}P_o$. By Proposition \ref{reversibility} the distribution of $(P_o, x_1)$ is the same as that of $(g_{x_1}P_o, x_1) = (Q_o, z_{-1})$.  Hence $z_{-1}$ is a uniformly distributed random neighbor of $o$ in $Q_o$.

By the Markov property, the conditional distribution of $\left(x_{n}\right)_{n \ge 1}$ given $P_o$ and $x_1$, is that of a Delaunay random walk starting at $x_1$ on $P_o$.   Applying $g_{x_1}$ one obtains that the conditional distribution of $\left(z_n\right)_{n \ge 0}$ given $Q_o$ and $z_{-1}$, is that of a Delaunay random walk starting at $o$ on $Q_o$. In particular, $z_1$ is a uniformly distributed random neighbor of $o$ in $Q_o$ which is independent from $z_{-1}$ conditioned on $Q_o$.

By definition, the conditional distribution of $\left(x_{-n}\right)_{n \ge 0}$ given $P_o$ and $x_1$ is that of a Delaunay random walk starting at $o$ on $P_o$.   Applying $g_{x_1}$ one obtains that the conditional distribution of $\left(z_{-n}\right)_{n \ge 1}$ given $Q_o$ and $z_{-1}$ is a Delaunay random walk starting at $z_{-1}$ on $Q_o$.  This implies that $\left(z_{-n}\right)_{n \ge 0}$ is a Delaunay random walk starting at $o$ in $Q_o$ which is independent from $\left(z_n\right)_{n \ge 0}$.

Notice that $\left(d(x_m,x_n)\right)_{m,n \in \Z}$ has the same distribution as $\left(d(z_m,z_n)\right)_{m,n\in \Z}$, and by definition one has
\[d(z_m,z_n) = d(g_{x_1}x_{m+1},g_{x_1}x_{n+1}) = d(x_{m+1},x_{n+1})\]
where the last equality follows because $g_{x_1}$ is an isometry.  This shows that $\left(x_n\right)_{n \in \Z}$ is distance stationary as claimed.
\end{proof}

\part{Zero-one laws for Poisson--Delaunay random walks\label{ergodicitypart}}

We maintain the notation from Part \ref{stationaritypart} As before, $M$ denotes a Riemannian symmetric space, $o \in M$ a base point, and $P$ a homogeneous Poisson point process in $M$.  

We will now investigate various $0$--$1$ laws. Our primary motivations are to show that certain asymptotics of Poisson--Delaunay random walks are deterministic: throughout this section, we will let $(x_n)_{n \in \Z}$ be a Delaunay random walk on $P_o$ starting at $o.$  Before delving deeper, we introduce some of the statistics we would like to say are deterministic.

\subsubsection*{Graph speed}  
Let $d_G(x,y)$ denote the graph distance between $x,y \in P_o$ in the Delaunay graph of $P_o$.  
By stationarity of the Poisson-Delaunay random graph under the degree biased probability (see \cite[Theorem 1.4]{paquette} and \cite[Proposition 2.5]{benjamini-curien}) the graph speed
\[\ell_G = \lim\limits_{n \to +\infty}\frac{1}{n}d_G(x_0,x_n)\]
exists almost surely and in $L^1$.  Furthermore $\ell_G \le 1$ since $d_G(x_0,x_1) = 1$ almost surely.

\subsubsection*{Ambient speed} 
By \cite[Lemma 3.2, Theorem 3.3]{paquette} one has that $\deg(o)$ has moments of all order, and $d(x_0,x_1)$ has all exponential moments, therefore one obtains that $\Edg{d(x_0,x_1)} < +\infty$. 
Hence, by distance stationarity (Theorem \ref{stationarity}) the ambient speed
\[\ell = \lim\limits_{n \to +\infty}\frac{1}{n}d(x_0,x_n)\]
also exists almost surely and in $L^1$ under the degree biased probability. 

\subsubsection*{Asymptotic random walk entropy}
Let $p^n(x,y)$ denote the probability that a Delaunay random walk on $P_o$ starting at $x$ will arrive at $y$ after $n$ steps, conditional on $P_o$.   We define the asymptotic entropy
\[h = \lim\limits_{n \to +\infty}-\frac{1}{n}\log(p^n(o,x_n))\]
where $(x_n)_{n \ge 0}$ is a Delaunay random walk on $P_o$ starting at $o$.  The limit exists almost surely and in $L^1$ by Kingman's subadditive ergodic theorem.

%\subsubsection*{Exit measure}
%For $M$ which is additionally non--positively curved, for example when $M$ is of noncompact type, every point in $x_n$ is connected by a unique geodesic to $o.$  Hence, it is possible to define $\theta_n$ as the projection of the Poisson--Delaunay walk to the unit tangent sphere of $M$ at $o.$  Say the random walk is \emph{non--wandering} if $\theta_n$ converges almost surely as $n \to\infty$.  We will show shortly that
%\[
%\P( (x_n)_{n \geq 0} \emph{ is non--wandering } ) \in \{0,1\}.
%\]
%In the non--wandering case, we therefore have a well--defined limit variable $\theta_{\infty}.$  
%Fixing a subset $U \subset \text{UT}_o(M),$ we can ask if $\theta_\infty$ has positive probability of exiting in $U.$  We will see that
%\[
%\P\bigl[ \P( \theta_{\infty} \in U~\vert~\sigma(P_o)) > 0\bigr] \in \{0,1\}.
%\]

\subsection{Poisson spatial zero-one law}

Recall that $\discrete{M}$ denotes the space of discrete subsets of $M$.  The topology on $\discrete{M}$ is that of weak convergence of point measures, by this we mean that a neighborhood of a discrete set $X \in \discrete{M}$ may be defined by picking an open set $U \subset M$ whose boundary is disjoint from $X$, a positive number $\epsilon$, and considering all discrete subsets $X'$ with the same number of points in $U$ as $X$ and such that the Hausdorff distance between $X' \cap U$ and $X \cap U$ is less than $\epsilon$.

Let $\pi_r:\discrete{M} \to \discrete{M}$ be the mapping $\pi_r(X) = X \setminus B_r$ where $B_r$ is the open ball centered at $o$ with radius $r$.  Let $\F^r$ be the $\sigma$-algebra of Borel subsets of $\discrete{M}$ generated by $\pi_r$.   We define the tail $\sigma$-algebra by the equation $\F^{\infty} = \bigcap\limits_{r > 0}\F^r$.

Informally the $\sigma$-algebra $\F^r$ only allows one to distinguish events that happen outside of the ball $B_r$ while the tail sets are characterized by properties of a discrete set $X \in \discrete{M}$ which do not depend on any bounded subset of $X$.  For example, the family of discrete subsets such that $\lim_{r \to +\infty} |X \cap B_r|/V(r)$ exists is a tail subset.

\begin{lemma}[Spatial zero-one law]\label{spatialzeroone}
All tail Borel subsets of $\discrete{M}$ are trivial for the Poisson point process (that is they have probability equal to either $0$ or $1$).
\end{lemma}
\begin{proof}
This is a corollary of Kolmogorov's zero-one law.  

To see this let $A_n = B_{n+1} \setminus B_n$, and notice that if $P$ is a Poisson process then the point processes $P_n = P \cap A_n$ are independent.   Any event of the form $\lbrace P \in T\rbrace$ with $T$ a tail subset of $\discrete{M}$ belongs to the tail $\sigma$-algebra of the sequence $P_n$ and is therefore trivial by Kolmogorov's zero-one law.
\end{proof}

Two graphs $X$ and $X'$ are said to be finite perturbations of one another if there exist finite subsets $K \subset X$ and $K' \subset X'$ such that $X \setminus K$ and $X' \setminus K'$ are isomorphic (here one removes the sets $K$ and $K'$ and all edges having an endpoint in them).  By Lemma \ref{noemptyhoroballs} the following result implies that any property of the Poisson-Delaunay graph which is stable under finite perturbations, is a tail event for the underlying Poisson process.

\begin{lemma}\label{poissonfiniteperturbation}
If two discrete sets which intersect all horoballs, coincide outside of a bounded subset of M then their Delaunay graphs are finite perturbations of one another.
\end{lemma}
\begin{proof}
Let $X$ and $X'$ be two discrete subsets of $M$ with bounded Voronoi cells and
$r > 0$ be such that $X \setminus B_r = X' \setminus B_r$.

By a Voronoi flower of a point $x$ in $X$ we mean an open set containing an open disk which is disjoint from $X$ and contains $x$ and $y$ on its boundary for each Delaunay neighbor $y$ of $x$ in $X$.  Similarly we define a Voronoi flower for a point in $X'$. See Figure \ref{flowerfigure}.

\begin{figure}
\centering
\includegraphics[scale=0.5]{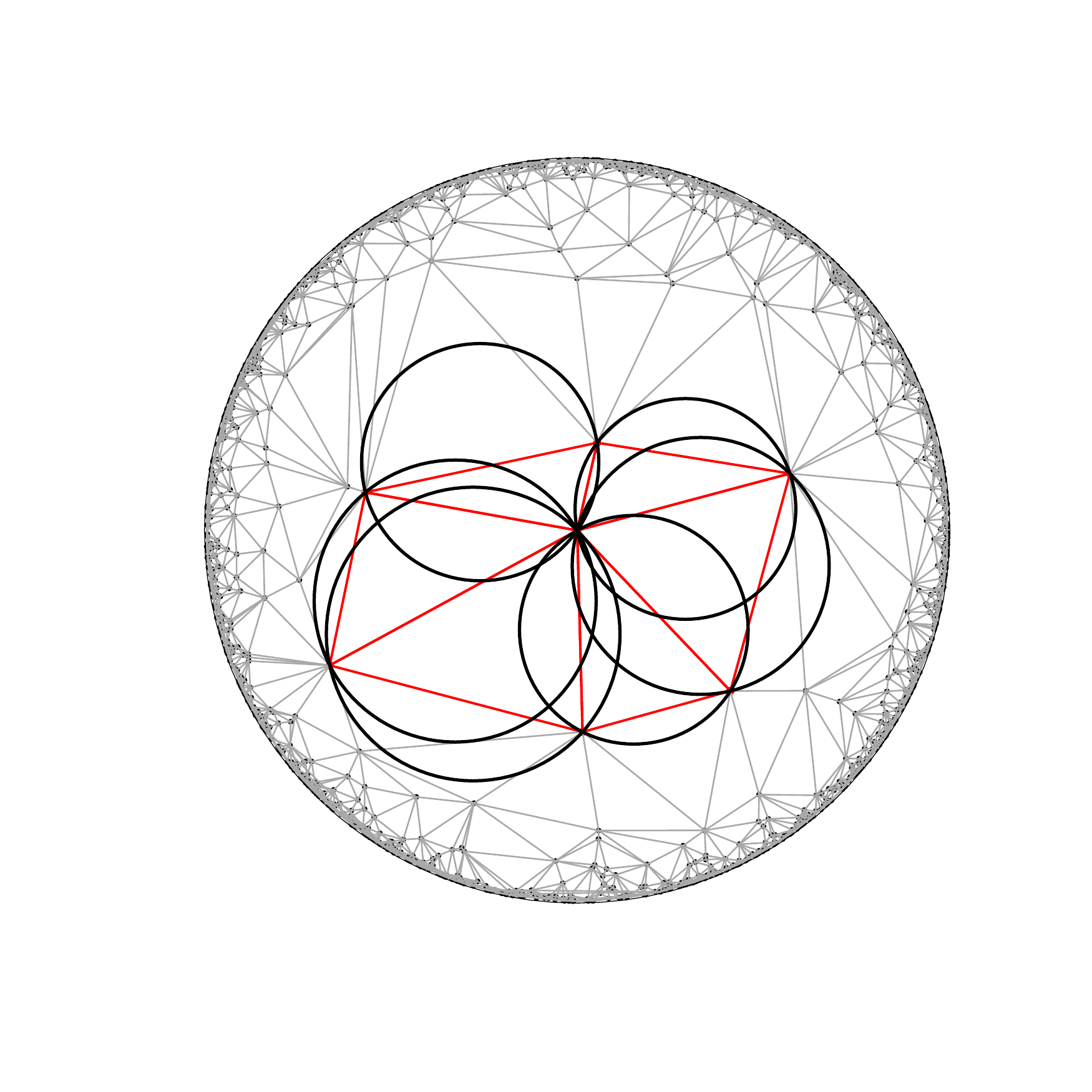}
\caption{\label{flowerfigure} A Voronoi flower in the hyperbolic plane.}
\end{figure}

Let us say a point $x \in X$ is good if there is a Voronoi flower for $x$ which is disjoint from $B_r$.  Similarly we say a point $x \in X'$ is good if it admits a Voronoi flower with respect to $X'$ which is disjoint from $B_r$.

We claim that that the set of good points with respect to $X$ and $X'$ coincide.  To see this notice that if $x \in X$ is good, it must lie outside of $B_r$ and furthermore all its neighbors do as well.  Hence $x$ and all its neighbors are in $X'$.  Furthermore the Voronoi flower for $x$ with respect to $X$ which is disjoint from $B_r$ is also a Voronoi flower for $x$ with respect to $X'$.  This establishes the claim (by symmetry).

Furthermore notice that if two good points are Delaunay neighbors in $X$ they are also Delaunay neighbors in $X'$.

Hence to establish the lemma it remains only to show that the set of good points has a finite complement in $X$ and $X'$.  We give the argument for $X$ only since the claim is symmetric with respect to exchanging $X$ and $X'$.

Suppose for the sake of contradiction that there exists a sequence of distinct points $x_n \in X$ which are not good. Then for each $n$, every Voronoi flower of $x_n$ intersects $B_r$.  In particular, for each $n$,  some disk with $x_n$ (and a certain Delaunay neighbor) on its boundary must be disjoint from $X$ and intersect $B_r$.   The existence of such a sequence of disks contracts the fact that $X$ intersects all horoballs.   Hence the set of good points is cofinite in $X$ (and by the same argument also in $X'$).
\end{proof}

\subsection{Ergodicity}

Let $\FS \subset \mathcal{F}$ be the $\sigma$-algebra of angularly invariant events, that is $A \in \FS$ implies that 
for all isometries $\psi$ of $M$ in the stabilizer of $o,$
\[
\Pdg( 
\{(P_o, (x_n)_{n \in \Z}) \in A\}
\triangle
\{(\psi(P_o), (\psi(x_n))_{n \in \Z}) \in A\})=0.
%(\Pi_o, (y_k)_{k \in \Z}) \in A 
%\Longleftrightarrow
%(\psi(\Pi_o), (\psi(y_k))_{k \in \Z}) \in A.
\]

For each \((X,(y_n)_{n \in \Z}) \in \discrete{M}\times M^\Z\) such that \(y_n \in X\) for all \(n\), and \(y_0 = o\), define
\[T(X,(y_n)_{n \in \Z}) = (g_{y_0}X, (z_n)_{n \in \Z})\]
where \(z_n = g_{y_1}(x_{n+1})\) for all \(n\) and \(g_x\) denotes the central symmetry exchanging \(o\) and \(x\).

Notice that by Theorem~\ref{stationarity}, the transformation \(T\) preserves the distribution of \((P_o,(x_n)_{n \in \Z})\).  We will show that this measure is ergodic for the restriction of \(T\) to \(\FS\).

\begin{theorem}[Ergodicity]\label{ergodicity}
Let $\left(x_n\right)_{n \in \Z}$ be a Delaunay random walk starting at $o$ on $P_o$ where $P$ is a stationary Poisson point process on a Riemannian symmetric space $M$ with base point $o$.  With $T$ as above there are $\Pdg$--nontrivial $T$--invariant events in $\FS$ if and only if $M$ is compact.
\end{theorem}
\begin{remark}
We do not use any special features of $P$ being Poisson.  This theorem also holds for tail--trivial distributional lattices, as defined in \cite{paquette}.
\end{remark}

Before proceeding to the proof, we give a corollary.
\begin{corollary}\label{ergodicitycorollary}
For a Poisson--Delaunay random walk on a Riemannian symmetric space, $\ell, \ell_G,$ and $h$
%h$ and the existence of $\theta_\infty$ (being non--wandering) 
are all deterministic.  
\end{corollary}
\begin{proof}
We discuss some details of the statement for $\ell.$  Similar arguments show the claim for the other quantities.
Recall that $\ell$ is the limit 
\[\ell = \lim\limits_{n \to +\infty}\frac{1}{n}d(x_0,x_n),\]
whose existence was guaranteed to exist by the subadditive ergodic theorem.  Observe that
\[
\ell =\lim\limits_{n \to +\infty}\frac{1}{n}d(x_1,x_{n})=\lim\limits_{n \to +\infty}\frac{1}{n-1}d(x_1,x_{n})
=\lim\limits_{n \to +\infty}\frac{1}{n-1}d(g_{x_1}(x_1), g_{x_1}(x_{n})).
\]
So $\ell$ is $T$--invariant.  As $\ell$ carries no angular information, we have that $\ell$ is deterministic when $M$ is noncompact.  If $M$ is compact, then $\ell = 0$ as the diameter of the manifold is finite.
\end{proof}

\begin{proof}[Proof of Theorem~\ref{ergodicity}]
  If $M$ is compact, then the number of points in $P_o$ has the distribution of $1+X$ where $X$ is Poisson with mean $\lambda \cdot \operatorname{Vol}(M).$  Hence $|P_o| = k$ for any $k > 0$ is a non--trivial $T$--invariant and angularly invariant event.

  If $M$ is noncompact, then the number of points in $P_o$ is almost surely infinite.  Hence, its Poisson--Delaunay graph is infinite.  Let $E \in \FS$ be an arbitrary $T$--invariant event, so that $\Pdg({E} \triangle T({E})) = 0.$  
  
  For $k \in \Z,$ let $\mathcal{F}_k = \sigma( (x_n)_{ n \leq k}, \mathcal{P}_o).$
  The martingale $Z_k = \Edg{ \one{E} \vert \mathcal{F}_k}$ is uniformly integrable, and hence $Z_k \to \one{E}$ as $k \to \infty$ almost surely.  On the other hand, using that $T$ is measure preserving, we have an equality in distribution
  \[
    \Edg{ \one{E} \vert \mathcal{F}_k} ( (P_o, (x_n)_{n \leq k}))
    \lawequals
    \Edg{ \one{E} \vert \mathcal{F}_k} ( (g_{x_1} P_o, g_{x_1}(x_n)_{n \leq k+1})).
  \]
  Changing variables, we can express
  \[
    \Edg{ \one{E} \vert \mathcal{F}_k} ( (g_{x_1} P_o, g_{x_1}(x_n)_{n \leq k+1}))
    =
    \Edg{ \one{T^{-1}(E)} \vert \mathcal{F}_{k+1}} ( (P_o, (x_n)_{n \leq k+1})).
  \]
  Using invariance of $E,$ we conclude that
  \[
    \Edg{ \one{E} \vert \mathcal{F}_k} ( (P_o, (x_n)_{n \leq k}))
    \lawequals
    \Edg{ \one{E} \vert \mathcal{F}_{k+1}} ( (P_o, (x_n)_{n \leq k+1})).
  \]
  Therefore, on taking $k \to \infty,$ we conclude that $\Edg{ \one{E} \vert \mathcal{F}_k} \in \{0,1\}$ almost surely, for each $k \in \Z,$ which implies that $E$ is measurable with respect to $\mathcal{F}_k$ up to modification by a $\Pdg$-null set.  
  
  The same argument shows that $\Edg{ \one{E} \vert \sigma( P_o, (x_n)_{n \geq k})} \in \{0,1\}$ almost surely, and so $E$ is measurable with respect to $\sigma( P_o, (x_n)_{n \geq k})$.  As the left and right tails of $(x_n)_{n \in \Z}$ are independent given $P_o,$ it must be that $E$ is in $\sigma(P_o)$ up to modification by a $\Pdg$--null set.  
  
  In particular, there is some Borel set $A$ in $\discrete{M}$ so that $E = \{P_o \in A\}$ up to $\Pdg$--null events.  Invariance of $E$ implies that for each $n \in \Z,$
  \[
  \one{ T^n( P_o, (x_k)_{x \in \Z}) \in E}
  =
  \one{ P_o \in A}
  \]
  $\Pdg$--a.s.  
  
  For any path $p= u_0 u_1 u_2 \dots u_{k-1} u_k$  with $u_0 = o$ in the Delaunay graph, let $g^*_p$ be the isometry of $M$ defined inductively by
  \[
  g^*_p =  g_{g^*_{q}(u_k)} \circ g^*_q
  \]
  where $q$ is the path $u_0 u_1 u_2 \dots u_{k-1}$ and $g^*_{u_0} = \operatorname{Id}.$  Observe that $g^*_p$ is an isometry that takes $u_k$ to $o,$ and therefore that $g^*_p \circ g_{u_k}^{-1}$ is in the stabilizer of $o.$
  
  Taking conditional expectations with respect to $\sigma(P_o),$ we can write
  \[
  \begin{aligned}
  \one{ P_o \in A}
  &=
  \Edg{ T^n( P_o, (x_k)_{x \in \Z}) \in E \vert \sigma(P_o)} \\
  &= \sum_{p} \one{g^*_p(P_o) \in A} \Pdg( (x_k)_{k=0}^n = p \vert \sigma(P_o))
  \end{aligned}
  \]
  $\Pdg$--a.s., where the sum is over all paths $p$ started from $o$ in the Delaunay graph on $P_o.$  As these are indicators and $\Pdg( (x_k)_{k=0}^n = (\cdot) \vert \sigma(P_o))$ is a probability measure on paths, it follows that
  \[
    \one{ P_o \in A}
    =
    \one{g^*_p(P_o) \in A}
  \]
  $\Pdg$--a.s.\,for all paths $p$ of length $n$ started at $o.$  As $n$ is arbitrary, and each point in $P_o$ can be reached with positive probability by $(x_n)_{n \in \Z},$ we conclude by angular invariance of $A$ that
  \[
    \one{ P_o \in A}
    =
    \one{g_u(P_o) \in A}
  \]
  $\Pdg$--a.s.\,for all $u \in P_o.$  As this holds $\Pdg$--a.s.\,it also holds $\P$--a.s.
  
  Fix $K > 0.$  For any $\epsilon > 0,$ we may approximate $A$ by $A'$ in the Borel algebra of $\discrete{B_R(o)}$ with $\P( \{P_o \in A'\} \triangle \{P_o \in A\}) < \epsilon$ for $R$ sufficiently large.  Let $y \in P_o$ be the point minimizing $d(o,y)$ among $P_o \setminus B_{R+K}(o).$  Then by stationarity of $P$
  \[
  \P( \{g_y(P_o) \in A'\} \triangle \{g_y(P_o) \in A\}) < \epsilon
  \]
  Hence by invariance of $A$
  \[
  \P( \{g_y(P_o) \in A'\} \triangle \{P_o \in A\}) < \epsilon
  \]
  Therefore, we have approximated $A$ by an event measurable with respect to $\mathcal{F}^K.$  As $\epsilon$ and $K$ were arbitrary, it follows that $\{P_o \in A\}$ is in the tail $\mathcal{F}^{\infty}$ up to modification by a null--set.
  \end{proof}

\part{Graph and ambient speed of Poisson-Delaunay random walks\label{speedcomparisonpart}}

We maintain the notation of Part \ref{stationaritypart} and Part \ref{ergodicitypart} but restrict ourselves from now on to the case where \(M\) is a Riemannian symmetric space of non-compact type.

The point of this section is to show that the ambient speed \(\ell\) and graph speed \(\ell_G\) of the Poisson-Delaunay random walk are zero or positive simultaneously.

\section{Graph vs ambient speed}\label{sc:graphvsambientspeed}

\begin{proposition}[Graph and ambient speed comparison]\label{graphvsambientspeed}
For any Poisson-Delaunay random walk on a Riemannian symmetric space of non-compact type one has $\ell_G = 0$ almost surely if and only if $\ell = 0$ almost surely.
\end{proposition}
\begin{proof}
We will begin by showing that if $\ell_G = 0$ almost surely then $\ell = 0$ almost surely.

By \cite[Proposition 4.1]{paquette} there exist positive constants $t_0$ and $\delta$ depending only on $M$ and $\lambda$ such that
\[\Pof{B^G_n  \not\subseteq B_{tn}} \le e^{-ne^{\delta t}}\]
for all $n=1,2,\ldots$ and $t > t_0$, where $B^G_n$ denotes the graph ball of radius $n$ centered at $o$ in $P_o$ with respect to $d_G$, and $B_r$ the ball of radius $r$ centered at $o$ in $M$ with respect to $d$.

By the Borel-Cantelli Lemma one has for any fixed $t > t_0$ that $B_n^G \subseteq B_{tn}$ for all $n$ large enough almost surely.  This implies that $t\ell_G \ge \ell$ almost surely.  Hence, we have shown that if $\ell_G = 0$ almost surely then $\ell = 0$ almost surely as claimed.

We will now show that if $\ell = 0$ almost surely then $\ell_G = 0$ almost surely.   

Recall that the Poisson-Delaunay graph is stationary under degree biased probability.  Furthermore since $B_n^G \subseteq B_{tn}$ for all $n$ large enough one has that
\[\limsup\limits_{n \to +\infty}\frac{1}{n}\log\left(|B_n^G|\right) \le \lim\limits_{n \to +\infty}\frac{1}{n}\log\left(|P_o \cap B_{tn}|\right) < +\infty\]
so the Poisson-Delaunay graph has finite exponential growth almost surely.

By \cite[Lemma 5.1]{carrascopiaggio2016} one has $\frac{1}{2}\ell_G^2 \le h$ almost surely.  Hence, it suffices to establish that $h = 0$ almost surely to obtain that $\ell_G = 0$ almost surely.  We will in fact show that the conditional expectation of $h$ given $P_o$ is $0$, which suffices because $h$ is non-negative.   

From $L^1$ convergence one obtains
\[\E{h|P_o} = \lim\limits_{n \to +\infty}\frac{1}{n}\sum\limits_{x \in P_o}-p^n(o,x)\log(p^n(o,x))\]
almost surely.

Since $\ell = 0$ one may choose a deterministic sequence $r_n \to +\infty$ such that $r_n = o(n)$ such that $p_n = \Pof{x_n \in B_{r_n}} \to 1$ when $n \to +\infty$.  

Conditioning on the event that $x_n \in B_{r_n}$, and using the fact that the entropy of a random variable is at most the logarithm of the number of distinct possible values, one obtains
\[\sum\limits_{x \in P_o}-p^n(o,x)\log(p^n(o,x)) \le p_n\log\left(|P_o \cap B_{r_n}|\right) + (1-p_n)\log\left(|B_n^G|\right).\]

To bound the first term on the right hand side notice that $|P_o \cap B_r|/V(r) \to \lambda$ almost surely when $r \to +\infty$, where $V(r)$ denotes the volume of the ball of radius $r$ in $M$.  This implies that $|P_o \cap B_{r_n}| = (1+ o(n))\lambda V(r_n)$.  

Finally, since $V(r) \le e^{b r}$ for all $r$ large enough one obtains that $\log(|P_o \cap B_{r_n}|) = O(r_n) = o(n)$.

For the second term on the right hand side above we use the previously established fact that $\frac{1}{n}\log\left(|B_n^G|\right) = O(1)$, which immediately implies (since $1-p_n$ goes to $0$) that the term is $o(n)$.

Hence we have shown that $h = 0$ almost surely from which $\ell_G = 0$ almost surely as claimed.\end{proof}

\part{Hyperbolic Poisson-Delaunay random walks\label{speedasymptoticspart}}

The purpose of this part of the article is to establish an estimate for the speed of Poisson-Delaunay random walks in hyperbolic space when the intensity of the point process is small.

In what follows $\H^d$ denotes $d$-dimensional hyperbolic space, and $o$ some fixed base point.  We assume that we have, defined on the same probability space, for each $\lambda > 0$ a Poisson point process $P_\lambda$ on $\H^d$ with intensity $\lambda$.  One way to do this is to let $P$ be a unit intensity Poisson process on $\R \times \H^d$ and let $P_\lambda$ be the projection onto $\H^d$ of $P \cap [0,\lambda]\times \H^d$ (this is a `Poisson rain' process as discussed in \cite[pg. 57]{kingman}).

For each $\lambda$ we let $\Plambda$ be the degree biased probability defined by $P_\lambda$, and use $\Elambda{\cdot}$ to denote expectation relative to this probability.

We assume that, on the same probability space, there are defined for each $\lambda$ a Delaunay random walk $(x_{n,\lambda})_{n \in \Z}$ on $P_\lambda \cup \lbrace o\rbrace$ starting at $o$.  And that there exists a random variable $u$ which is uniformly distributed on $[0,1]$ and independent from all the previously defined random objects.

Let $\ell_\lambda$ denote the speed of $(x_{n,\lambda})_{n \in \Z}$ and $\ell_{G,\lambda}$ its graph speed.  By Corollary \ref{ergodicitycorollary} both speeds are almost surely constant.

We will fix from now on for each $\lambda$ a random horofunction $\xi_\lambda$ given by the Furstenberg type formula for speed established in Theorem \ref{furstenbergtheorem} applied to the sequence $(x_{n,\lambda})_{n \in \Z}$.

\section{Speed asymptotics for low intensity}

We will now state our main result and give the proof asuming some results which will be established later on.
\begin{theorem}[Speed asymptotics for low intensity]\label{speedestimatestheorem}
The speed of the Poisson-Delaunay random walk on $\H^d$ is almost surely constant for each \(\lambda\) and satisfies the following asymptotic:
\[\lim\limits_{\lambda \to 0}\frac{\ell_\lambda}{\frac{2}{d-1}\log(\lambda^{-1})} = 1.\]
\end{theorem}
\begin{proof}
By Corollary \ref{ergodicitycorollary} the speed \(\ell_\lambda\) is almost surely constant.  

By the Furstenberg type formula for speed established in Theorem \ref{furstenbergtheorem} we have, for each $\lambda$, a random horofunction $\xi_{\lambda}$ such that
\[\ell_\lambda = \Elambda{\ell_\lambda} = -\Elambda{\xi_{\lambda}(x_{1,\lambda})}.\]

Since all horofunctions are $1$-Lipschitz one obtains
\[\left|\Elambda{\xi_{\lambda}(x_{1,\lambda})}\right| \le \Elambda{d(o,x_{1,\lambda})}.\]

We will show in Theorem \ref{momentstheorem} that the right hand side is equivalent to $\frac{2}{d-1}\log(\lambda^{-1})$ when $\lambda \to 0$.

To show that this upper bound is nearly optimal when $\lambda$ is small we write
\[-\Elambda{\xi_{\lambda}(x_{1,\lambda})} = \Elambda{d(o,x_{1,\lambda})} - \Elambda{\xi_{\lambda}(x_{1,\lambda})+d(o,x_{1,\lambda})}.\]

It remains to show that the second term on the right hand side is small relative to the first one.

For this purpose first notice that, since $\xi(x) + d(o,x) \ge 0$ for all horofunctions $\xi$ one has
\[0 \le \Elambda{\xi_{\lambda}(x_{1,\lambda})+d(o,x_{1,\lambda})}.\]

To obtain an upper bound for this expected value, we must first show that $\xi_\lambda$ is almost surely a boundary horofunction for each fixed $\lambda$.

To see this notice that, since the Delaunay graph of $P_\lambda \cup \lbrace o\rbrace$  is connected and infinite, the simple random walk on it cannot be positively recurrent (i.e. spend a positive fraction of its time in a finite set of vertices).   Also, since this graph embedded in $\H^d$ with only finitely many vertices in each bounded subset the claim follows from Proposition \ref{corollaryboundary}.

Setting $f_\xi(x) = \xi(x) + d(o,x)$, it is now possible to use the following worst case bound
\[\Elambda{\xi_{\lambda}(x_{1,\lambda})+d(o,x_{1,\lambda})} \le \E{|N_\lambda|}^{-1}\E{\max\limits_{\xi}\sum\limits_{x \in N_\lambda}f_\xi(x)}\]
where the maximum is over all boundary horofunctions $\xi$, and $N_\lambda$ is the set of neighbors of $o$ in $P_\lambda \cup \lbrace o \rbrace$.

We will show in Lemma \ref{neighborlemma} that $\E{|N_\lambda|}^{-1} \le C\lambda$ for some constant $C > 0$ and all $\lambda$ small enough.   We will also prove later on in Theorem \ref{sumtheorem} that 
\[\E{\max\limits_{\xi}\sum\limits_{x \in N_\lambda}f_\xi(x)} = o\left(\lambda^{-1}\log(\lambda^{-1})\right),\]
when $\lambda \to 0$.  Combining these two results completes the proof.
\end{proof}

Combined with the comparison of graph and ambient speeds (Proposition \ref{graphvsambientspeed}), the theorem above yields an alternate proof that the hyperbolic Poisson-Delaunay random walk has positive graph and ambient speed almost surely if the intensity is small enough.   

A more general result (in particular valid for all intensities) is established in \cite{paquette} by showing that the Delaunay graph is invariantly non-amenable and using the theory of unimodular random graphs.  Here we rely instead on the distance stationarity of the random walk and the Furstenberg type formula for speed.  The overlap between the two proofs is the need for some estimates on the number of neighbors of the root, the distance to the neighbors, and some exponential bound on the growth of the Delaunay graph.

\begin{corollary}[Positive speed for low intensities]\label{speedcorollary}
For all $\lambda$ small enough both $\ell_{\lambda}$ and $\ell_{G,\lambda}$ are almost surely positive.
\end{corollary}

Theorem \ref{speedestimatestheorem} also allows one to show that the graph speed goes to $1$ (its maximum possible value) as the intensity goes to zero.  This answers a question posed in \cite{benjamini-paquette-pfeffer}.
\begin{corollary}[Graph speed for small intensities]\label{graphspeedcorollary}
For the Poisson-Delaunay random walk on $\H^d$, one has $\ell_{G,\lambda} \to 1$ as $\lambda \to 0$.
\end{corollary}
\begin{proof}
Recall that by Corollary \ref{ergodicitycorollary} the graph speed \(\ell_{G,\lambda}\) is almost surely constant for each \(\lambda\).

By \cite[Proposition 4.1]{paquette} for each $\lambda > 0$ there exists $t_\lambda > 0$ 
such that, almost surely, the graph ball of radius $n$ centered at $o$ in $P_\lambda \cup \lbrace o\rbrace$ is contained in the metric ball $B_{t_\lambda n}$ for all $n$ large enough.  From this one obtains $t_\lambda \ell_{G,\lambda} \ge \ell_\lambda$ almost surely.

Notice that the isoperimetric constant of $\H^d$ is $d-1$.  Therefore, by \cite[Proposition 4.1]{paquette}, for any positive $\alpha < d-1$ one may guarantee that $t_\lambda = \frac{2}{\alpha}\log(\lambda^{-1}) + o(\log(\lambda^{-1}))$ when $\lambda \to 0$.

Combining this with Theorem \ref{speedestimatestheorem} one obtains that $\ell_{G,\lambda} \to 1$ 
as claimed.
\end{proof}

\section{Delaunay edge estimates}

The purpose of this section is to prove the following result, which was used in the proof of Theorem \ref{speedestimatestheorem} (we use $f \sim g$ to mean that $f/g$ converges to $1$),

\begin{theorem}\label{momentstheorem}
  For each $\alpha > 0$ one has 
  \[\Elambda{d(o,x_{1,\lambda})^\alpha} \sim \left(\frac{2}{d-1}\log(\lambda^{-1})\right)^{\alpha}\]
  when $\lambda \to 0$.
\end{theorem}

\subsection{Connection probability estimates}

As a first step towards the proof of Theorem \ref{momentstheorem} we will estimate the connection probability of two points in the Poisson-Delaunay triangulation.  This will allow us to obtain the asymptotic behavior of the number of Delaunay neighbors of the root $o$ when the intensity goes to $0$ (which appears implicitly whenever one calculates an expected value with respect to the degree biased probability).

In what follows we use $x \sim y$ to mean that $x$ and $y$ are Delaunay neighbors in some discrete set under consideration, and $V(r)$ to denote the volume of the ball of radius $r$ in $\H^d$ (we use the convention that $V(r) = 0$ if $r$ is negative).

\begin{lemma}\label{neighborprobabilitylemma}
  There is a positive constant $r_1$ such that for all $\lambda >0$ one has
  \[
    e^{-\lambda V(r/2-r_1)}
    \geq
    \Pof{x \sim o \text{ in }P_\lambda \cup \lbrace o,x\rbrace}
    \geq e^{-\lambda V(r/2)},
  \]
  where $r = d(o,x)$.
\end{lemma}

The lower bound follows from the observation that if the open ball $W$, with diameter given by the geodesic segment $[0,x],$ contains no points of $P_\lambda,$ then $x \sim o$.  The probability that $W \cap P_\lambda = \emptyset$ is $e^{-\lambda V(r/2)}$ giving the lower bound.  

In order to bound from above the probability that a given point $x$ is a Delaunay neighbor of $o$ in $P_\lambda \cup \lbrace o,x\rbrace$, we will use the fact that this implies that the intersection of all balls containing both $o$ and $x$ is disjoint from $P_\lambda$.

Basic hyperbolic geometry implies that the volume of this set is of order $V(r/2)$ where $r = d(o,x)$.   The result could be obtained by applying \cite[Proposition 14]{chatterji-niblo} from which one obtains immediately that the set contains a ball of radius $r - 2\delta$ where $\delta > 0$ is the constant of hypberbolicity of $\H^d$.  We give an independent (more elementary) proof here.

\begin{lemma}[The intersection of balls containing two points is thick]\label{ballintersection}
There exists a positive constant $r_1$ such that for all $p,q$ in $\H^d$ the ball of radius $d(p,q)/2 - r_1$ centered at the midpoint $m$ of $p$ and $q$ is contained in all open balls having $p$ and $q$ on their boundary.
\end{lemma}
\noindent Before embarking on this proof, we note that this will complete the proof of Lemma~\ref{neighborprobabilitylemma}, since the fact that $o \sim x$ in $P_\lambda \cup \lbrace o,x\rbrace$ implies that a ball of volume $V(r/2 - r_1)$ is disjoint from $P_\lambda$.
\begin{proof}
We will show that the proposition is valid for any $r_1$ such that $\tanh(r_1)r_1 \ge \log(8)$ (for example $r_1 = 3$ will suffice).  In what follows $E(p,q)$ denotes the intersection of all open balls having $p$ and $q$ on their boundary.

First, observe that it is enough to prove the two dimensional case. In fact, suppose there is a point $z$ in a ball $B_r(m)$ which is not in $E(p,q)$. Consider the embedded hyperbolic plane $H$ passing through $p,q$ and $z$, and let $E_H(p,q)$ be the intersection of all hyperbolic disks of $H$ having $p$ and $q$ on their boundary. Since $z$ is not in $E(p,q)$, there exists a ball $B$ having $p,q$ on its boundary that does not contain $z$. The intersection $B\cap H$ is an Euclidean disk, and therefore, a hyperbolic disk of $H$ having $p$ and $q$ on its boundary that does not contain $z$. Also, the intersection $B_r^H(m)=B_r(m)\cap H$ is a hyperbolic disk centered at $m$ of radius $r$ in $H$. This shows that $z$ belongs to $B_r^H(m)\setminus E_H(p,q)$. Therefore, if the statement hold for some $r$ in $\H^2$, it also holds for the same $r$ in $\H^d$.

\begin{figure}
\centering
\includegraphics[scale=0.5]{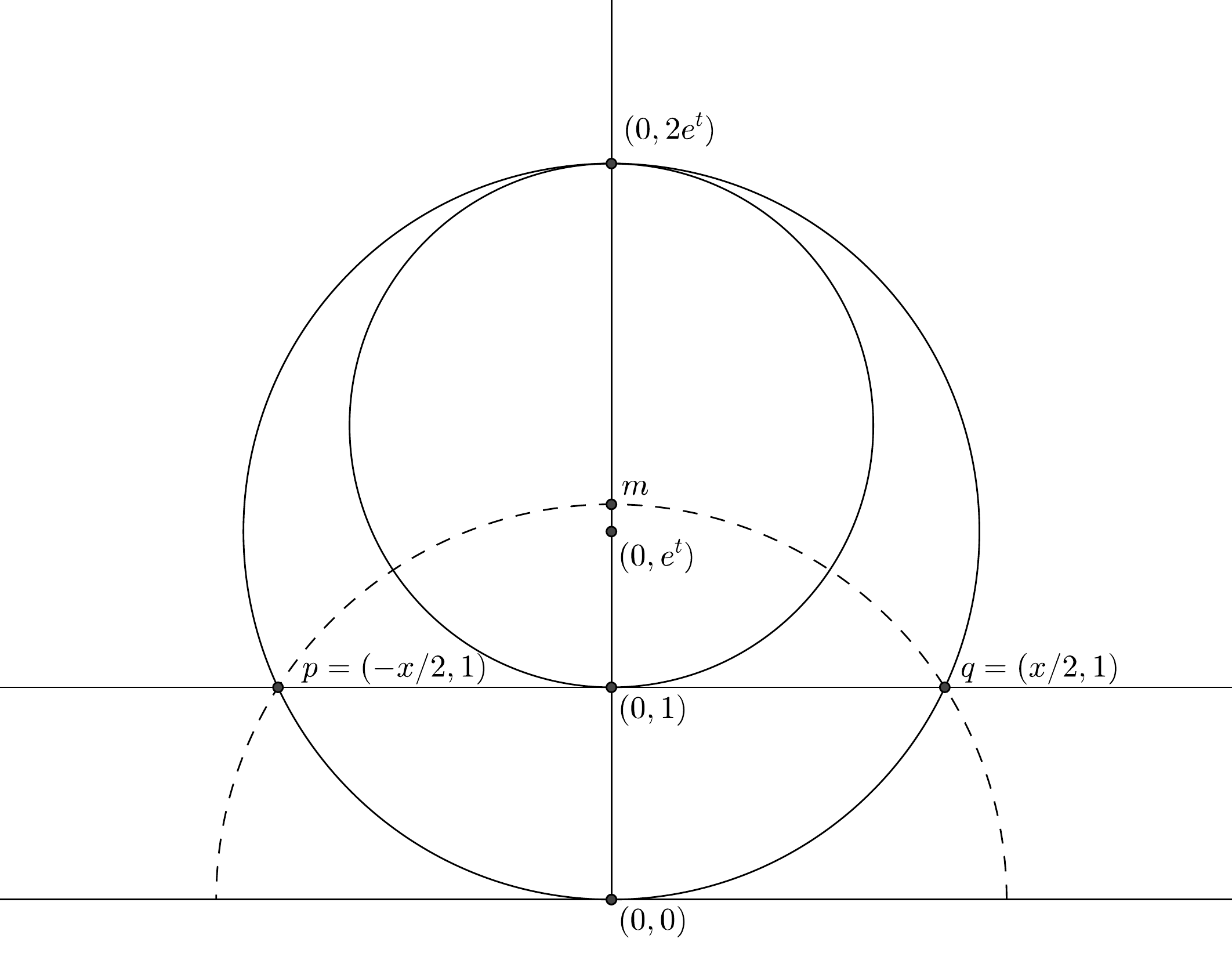}
\caption{\label{figuraEpq} The intersection of all balls containing the points $p$ and $q$ on their boundary, with $d(p,q)\geq 2r_1$, contains a ball of radius $d(p,q)/2-r_1$.}
\end{figure}

In the upper half plane model of the hyperbolic plane we assume from now on that the points are $p = (-x/2,1)$ and $q = (x/2,1)$. The distance $r = d(p,q)$ statisfies $\cosh(r) = 1 + x^2/2$. Using the fact that hyperbolic disks in the upper half plane model are simply Euclidean disks which do not meet the boundary, one obtains that $E(p,q)$ is the set of points with $y \ge 1$ contained in the Euclidean disk $D$ passing through $(0,0)$, $p$ and $q$. See Figure \ref{figuraEpq}.

Let $(0,e^t)$ be the center of $D$. The Euclidean disk $B$ whose diameter is the segment joining $(0,1)$ to $(0,2e^t)$ is contained in $E(p,q)$. But $B$ is also a hyperbolic disk, with the same segment as a diameter and centered at $m$. This follows since the inversion with respect to $C$, the circle passing through $p$ and $q$ which is ortogonal to the boundary of $\H^2$, exchanges $D$ and the horizontal line $\{y=1\}$, and thus
$$d((0,2e^t),m)=\log\left(2e^t/r_0\right)=\log\left(r_0/1\right)=d(m,(0,1)),$$
where $r_0$ is the Euclidean radius of $C$. The hyperbolic radius of $B$ is $(\log(2)+t)/2$.

We will now show that $t \ge r-r_1$. From the fact that $(0,e^t)$ is equidistant (with respect to the Euclidean distance) to $(0,0)$ and $(x/2,1)$, one obtains
\[2e^t = 1 + \frac{x^2}{4}.\]
From this it follows that  $\cosh(r)/\cosh(t) \le 8$.   But using the intermediate value theorem for $\log(\cosh(t))$, we have
\[\log(\cosh(r)) - \log(\cosh(r-r_1)) \ge \tanh(r-r_1)r_1 \ge \tanh(r_1)r_1 \ge \log(8).\]
This implies, since $\cosh$ is increasing on the positive reals, that $t \ge r-r_1$ as claimed.
\end{proof}

\subsection{Expected number of neighbors}

As a first application of Lemma \ref{neighborprobabilitylemma} we obtain the asymptotic behavior of the expected degree of the root in the Poisson-Delaunay graph as the intensity goes to zero.

Before proving the result, let us record some estimates on the behavior of the volume growth function $V(r)$ and its derivative $v(r) = V'(r)$.

First, observe that both $v(r)e^{-r(d-1)}$ and $V(r)e^{-r(d-1)}$ converge to positive constants as $r \to \infty.$  Further, both are continuous functions on $[0,\infty)$ which vanish only at $0$ as $r^d$ and $r^{d-1}$ respectively.  

Hence, $v(r)$ is bounded from below by a positive multiple of $V(r/2)v(r/2)$ on all its domain.  

From now on, when comparing positive functions $f$ and $g$ we will write $f \lesssim g$ to mean that $f$ is bounded from above by a positive multiple of $g$ in the domain under consideration.  We write $f \approx g$ when $f \lesssim g \lesssim f$.

In particular, we have established above that $v(r) \gtrsim V(r/2)v(r/2)$ on $[0,+\infty]$.

\begin{lemma}\label{neighborlemma}
For all $\lambda$ small enough one has
\[\E{|N_\lambda|} \approx \lambda^{-1}\]
where $N_\lambda$ is the set of Delaunay neighbors of $o$ in $P_\lambda \cup \lbrace o\rbrace$.
\end{lemma}
\begin{proof}

By Slivnyak's formula and Lemma \ref{neighborprobabilitylemma} one has
  \begin{equation*}
    \begin{aligned}
    \E{|N_\lambda|} &= \int \Pof{x \sim o\text{ in }P_\lambda \cup \lbrace o,x\rbrace} dx
\\  &\geq
    \lambda \int_0^\infty 
    e^{-\lambda V(r/2)}
    \cdot
    \,v(r)dr \\ 
    &\gtrsim
    \lambda \int_0^\infty 
    e^{-\lambda V(r/2)}
    \cdot
    \,V(r/2)v(r/2)d(r/2) \\
    &=
    \lambda^{-1} \int_0^\infty 
    u
    e^{-u}
    \,du =\lambda^{-1}.
    \end{aligned}
  \end{equation*}

  This establishes to the lower bound.
  
  For the upper bound we begin again using Slivnyak's formula and Lemma \ref{neighborprobabilitylemma} to obtain
    \begin{align*}
    \E{|N_\lambda|} 
    &\leq
    \lambda \int_0^\infty 
    e^{-\lambda V(r/2-r_1)}
    \cdot
    \,v(r)dr 
    \\ &= \lambda \int_0^{4r_1} v(r)dr + \lambda \int_{4r_1}^{+\infty}e^{-\lambda V(r/2 - r_1)}v(r)dr
    \\ &= \lambda V(4r_1) + \lambda \int_{2r_1}^{+\infty}e^{-\lambda V(r/2)}v(r + 2r_1)dr.
    \end{align*}
    
    The first term on the right hand side goes to $0$ with $\lambda$.   For the second term we use the fact that $v(r+2r_1) \lesssim V(r/2)v(r/2)$ on $[2r_1,+\infty)$ to obtain
    \begin{align*}
    \lambda \int_{2r_1}^{+\infty}e^{-\lambda V(r/2)}v(r + 2r_1)dr &\lesssim \lambda \int_{2r_1}^{+\infty} e^{-\lambda V(r/2)}V(r/2)v(r/2) dr/2
\\  &= \lambda \int_{V(r_1)}^{+\infty}e^{-\lambda u}u du \approx \lambda^{-1},
    \end{align*}
    which concludes the proof.
\end{proof}

\subsection{Annulus containing most neighbors}

The last tool we will need in order to prove Theorem \ref{momentstheorem} is an estimate showing that we can ignore neighbors outside of a neighborhood of the sphere of radius $R_\lambda = \frac{2}{d-1}\log(\lambda^{-1})$.

Recall that $N_\lambda$ denotes the set of Delaunay neighbors of $o$ in $P_\lambda \cup \lbrace o\rbrace$.  In what follows, given $M > 0$, we define
\[N_{\lambda,M} = \left\lbrace x \in N_\lambda: |d(o,x)-R_\lambda| < M\right\rbrace.\]

\begin{lemma}\label{annuluslemma}
For each $\alpha > 0$ and $\epsilon > 0$ there exists $M > 0$ such that
\[\E{\sum\limits_{x \in N_\lambda \setminus N_{\lambda,M}}d(o,x)^{\alpha}} \le \epsilon \lambda^{-1}\log(\lambda^{-1})^\alpha\]
for all $\lambda$ small enough.
\end{lemma}
\begin{proof}
By Slivnyak's formula and Lemma \ref{neighborprobabilitylemma} one has
\begin{align*}
\E{\sum\limits_{x \in N_\lambda \setminus N_{\lambda,M}}d(o,x)^{\alpha}} &= \lambda \int\limits_{|d(o,x)-R_\lambda| \ge M}d(o,x)^\alpha\Pof{x \sim o\text{ in }P_\lambda \cup \lbrace o,x\rbrace}dx
\\ & \le \lambda \int\limits_{|r-R_\lambda| \ge M}r^\alpha e^{-\lambda V(r/2 - r_1)} v(r)dr
\\ & \le \lambda (4r_1)^\alpha V(4r_1) +  \lambda \int\limits_{\substack{|r-R_\lambda| \ge M\\ r \ge 4r_1}}r^\alpha e^{-\lambda V(r/2 - r_1)} v(r)dr.
\end{align*}

The first term above goes to zero and therefore can be ignored.

To control the second term we first observe, repeating the argument from Lemma \ref{neighborlemma}, that
\begin{align*}
\lambda \int_{4r_1}^{R_\lambda -M} r^{\alpha}e^{-\lambda V(r/2 - r_1)}v(r) dr &\le \lambda R_\lambda^{\alpha}\int_{2r_1}^{R_{\lambda}-M-2r_1}e^{-\lambda V(r/2)}v(r+2r_1)dr
\\ & \lesssim \lambda R_\lambda^{\alpha}\int_{2r_1}^{R_{\lambda}-M-2r_1}e^{-\lambda V(r/2)}V(r/2)v(r/2)dr/2 
\\ & \lesssim \lambda R_\lambda^{\alpha}\int_{0}^{a_\lambda}e^{-\lambda u}udr 
\\ & \lesssim \lambda^{-1} \log(\lambda^{-1})^{\alpha}\int_{0}^{\lambda a_\lambda}e^{-u}udr,
\end{align*}
where $a_\lambda = V\left(\frac{1}{2}R_{\lambda}-\frac{1}{2}M-r_1\right)$.

Notice that $\lim\limits_{\lambda \to 0}\lambda a_\lambda$ exists and goes to $0$ when $M \to +\infty$.  Hence, given $\epsilon > 0$, one may choose $M$ sufficiently large so that
\[\lambda \int_{4r_1}^{R_\lambda -M} r^{\alpha}e^{-\lambda V(r/2 - r_1)}v(r) dr \le \frac{1}{2}\epsilon \lambda^{-1}\log(\lambda^{-1})\]
for all $\lambda$ small enough.

Similarly one has, using that $r \lesssim \log(V(r/2))$ for all sufficiently large, that
\begin{align*}
\lambda \int_{R_\lambda +M}^{+\infty} r^{\alpha}e^{-\lambda V(r/2 - r_1)}v(r) dr &\lesssim  \lambda \int_{R_\lambda +M- 2r_1}^{+\infty} \log(V(r/2))^{\alpha}e^{-\lambda V(r/2)}V(r/2)v(r)dr/2
\\ &= \lambda \int_{b_\lambda}^{+\infty} \log(u)^\alpha e^{-\lambda u} u du
\\ &= \lambda^{-1} \int_{\lambda b_\lambda}^{+\infty} \log(u/\lambda)^\alpha e^{-u} u du
\\ &\lesssim \lambda^{-1} \int_{\lambda b_\lambda}^{+\infty} (\log(u)^\alpha + \log(\lambda^{-1})^{\alpha}) e^{-u} u du,
\end{align*}
where  $b_\lambda = V\left(\frac{1}{2}R_{\lambda}+\frac{1}{2}M-r_1\right)$, and in the final step we have used that $(a+b)^\alpha \le C_\alpha (a^\alpha + b^\alpha)$ for all $a,b > 0$ and some constant $C_\alpha$ depending only on $\alpha$.

Once again, $\lim\limits_{\lambda \to 0}\lambda b_\lambda$ exists, but it goes to $+\infty$ when $M \to +\infty$.  Using this one obtains that, given $\epsilon > 0$, there exists $M$ large enough so that 
\[\lambda \int_{R_\lambda +M}^{+\infty} r^{\alpha}e^{-\lambda V(r/2 - r_1)}v(r) dr \le \frac{1}{2}\epsilon \lambda^{-1}\log(\lambda^{-1})^{\alpha}\]
for all $\lambda$ small enough.

Combining the results above we have shown that, given $\epsilon > 0$, one may choose $M$ large enough so that
\[\E{\sum\limits_{x \in N_\lambda \setminus N_{\lambda,M}}d(o,x)^{\alpha}} \le \epsilon \lambda^{-1}\log(\lambda^{-1})\]
for all $\lambda$ small enough, as claimed.
\end{proof}

\subsection{Proof of Theorem \ref{momentstheorem}}

To conlude the section we prove Theorem \ref{momentstheorem}.

Recall that $R_\lambda = \frac{2}{d-1}\log(\lambda^{-1})$, and that by Lemma \ref{neighborlemma}, there exists a constant $C> 0$ such that $C^{-1}\lambda^{-1} \le \E{|N_\lambda|} \le C \lambda^{-1}$.  Also, given $M > 0$ we defined $N_{\lambda,M}$ to be the set of neighbors of $o$ which are at distance between $R_\lambda - M$ and $R_\lambda + M$ from $o$.

By definition of the degree biased probability one has
\[\Elambda{d(o,x_{1,\lambda})^\alpha} = \E{|N_\lambda|}^{-1}\E{\sum\limits_{x \in N_\lambda}d(o,x)^{\alpha}}.\]

For the lower bound notice at least $|N_\lambda| - |P_\lambda \cap B_{R_\lambda}|$ neighbors of $o$ are at distance greater than $R_\lambda$ from $o$.  Combined with the bounds on $\E{|N_\lambda}$ above one obtains
\begin{align*}
\E{|N_\lambda|}^{-1}\E{\sum\limits_{x \in N_\lambda}d(o,x)^{\alpha}} &\ge R_\lambda^\alpha -\E{|N_\lambda|}^{-1}\E{|P_\lambda \cap B_{R_\lambda}|}
\\ &\ge  R_\lambda^\alpha - C^{-1}\lambda^{2} V(R_\lambda)
\\ &= R_\lambda^\alpha + O(1)
\end{align*}

For the upper bound, notice that for all $M > 0$ one has
\begin{align*}
\E{|N_\lambda|}^{-1}\E{\sum\limits_{x \in N_\lambda}d(o,x)^{\alpha}}& \le (R_\lambda + M)^\alpha + \E{|N_\lambda|}^{-1}\E{\sum\limits_{x \in N_\lambda \setminus N_{\lambda,M}}d(o,x)^\alpha}\\
& \le (R_\lambda + M)^\alpha + C^{-1}\lambda\E{\sum\limits_{x \in N_\lambda \setminus N_{\lambda,M}}d(o,x)^\alpha}
\end{align*}

By Lemma \ref{annuluslemma}, given $\epsilon>0$ one may choose $M$ so that 
\[C^{-1}\lambda\E{\sum\limits_{x \in N_\lambda \setminus N_{\lambda,M}}d(o,x)^\alpha} \le C^{-1}\epsilon \log(\lambda^{-1})^\alpha\]
for all $\lambda$ small enough.  

This shows that
\[\limsup\limits_{\lambda \to 0}\frac{\Elambda{d(o,x_{1,\lambda})^\alpha}}{R_\lambda^\alpha} \le 1 + C^{-1}\frac{(d-1)^\alpha}{2^\alpha}\epsilon,\]
for all $\epsilon > 0$.  From which the theorem follows immediately.

\section{Horofunction sums on the set of neighbors}

The purpose of this section is to complete the proof of Theorem \ref{speedestimatestheorem} by establishing the following (recall that $f = o(g)$ means that $f/g$ converges to $0$):

\begin{theorem}\label{sumtheorem}
 One has
\[\E{\max\limits_{\xi}\sum\limits_{x \in N_\lambda}f_\xi(x)} = o\left(\lambda^{-1}\log(\lambda^{-1})\right)\]
 when $\lambda \to 0$, where the maximum is over all boundary horofunctions and $f_\xi(x) = \xi(x) + d(o,x)$.
\end{theorem}
\begin{proof}
Recall $N_\lambda$ denotes the set of Delaunay neighbors of $o$ in $P_\lambda \cup \lbrace o\rbrace$ and that for each $M > 0$ we have defined $N_{\lambda,M}$ as the set of neighbors $x \in N_\lambda$ with $|d(o,x) - R_\lambda| < M$, where $R_\lambda = \frac{2}{d-1}\log(\lambda^{-1})$.

Notice that, for any $M > 0$, we may split the sum and use $f_\xi(x) \le 2d(o,x)$ to obtain
\[\E{\max\limits_{\xi}\sum\limits_{x \in N_\lambda}f_\xi(x)} \le 2\E{\sum\limits_{x \in N_\lambda \setminus N_{\lambda,M}}d(o,x)} + \E{\max\limits_{\xi}\sum\limits_{x \in N_{\lambda,M}}f_\xi(x)}\]

Let $\epsilon > 0$ be fixed from now on.   By Lemma \ref{annuluslemma} there exists $M$ such that the first term on the right hand side above is bounded by $\epsilon \lambda^{-1}\log(\lambda^{-1})$.

To bound the second term we will split it into a sum on points belonging to a small cone, and points where $f_\xi$ is small.

To make this precise denote by $D_\theta(v)$ the geodesic cone with radius $\theta > 0$ with vertex at $o$ and direction $v$, where $v$ is a unit tangent vector at $o$.   By this we mean the set of points of the form $\exp_o(tw)$ for some $t > 0$ and some unit tangent vector $w$ at $o$ forming an angle less than $\theta$ with $v$ (here $\exp_o$ denoting the Riemannian exponential map at $o$).

In Lemma \ref{conesetlemma} we will show there exists a function $r \mapsto \theta_r$, satisfying $\theta_r \approx e^{-r/2}$ when $r \to +\infty$, such that for each boundary horofunction $\xi$ the set of points where $f_\xi > r$ is contained in a some cone of the form $D_{\theta_r}(v_\xi)$.

Applying this to $r = \epsilon R_\lambda$, and splitting the sum among points where $f_\xi > r$ and the rest, one obtains
\[\E{\max\limits_{\xi}\sum\limits_{x \in N_{\lambda,M}}f_\xi(x)} \le \epsilon R_\lambda \E{|N_\lambda|} + 2(R_\lambda+M)\E{\max\limits_{v \in S^{d-1}}|P_\lambda \cap D_{\theta_r}(v) \cap B_{R_\lambda + M}|},\]
where abusing notation slightly $S^{d-1}$ denotes the unit tangent sphere at $o$, and recall that $B_r$ denotes the ball of radius $r$ centered at $o$.

By the definition of $R_\lambda$ and Lemma \ref{neighborlemma}, the first term is bounded by $\frac{2C}{d-1}\epsilon \lambda^{-1}\log(\lambda^{-1})$
for all $\lambda$ small enough, where the constant $C$ does not depend on $\epsilon$.

Notice that projecting the points of $P_\lambda \cap B_{R_\lambda + M}$ onto the unit sphere $S^{d-1}$ at $o$ along geodisic rays one obtains a Poisson point process on $S^{d-1}$ with intensity $\mu = \lambda V(R_\lambda + M)$ times the normalized volume.   This means that the quantity $\max\limits_{v \in S^{d-1}}|P_\lambda \cap D_{\theta_r}(v) \cap B_{R_\lambda + M}|$ can be interpreted as the maximum number of points of such a Poisson process which can be found in a metric ball with radius $\alpha = \theta_r$.

In this situation we will show in Lemma \ref{vapniklemma} that, as long as $\mu\alpha^{4(d-1)}$ remains bounded away from zero, the expected number of such points is bounded by a constant multiple of $\mu\alpha^{d-1}$ when $\mu \to +\infty$.   In our case this applies if $\epsilon < 1/4$ and yields that 
\[\E{\max\limits_{v \in S^{d-1}}|P_\lambda \cap D_{\theta_r}(v) \cap B_{R_\lambda + M}|} \lesssim \lambda^{-(1-\epsilon)}\]
when $\lambda \to 0$.

Combining these results one obtains that
\[\limsup\limits_{\lambda \to 0}\frac{1}{\lambda^{-1}\log(\lambda^{-1})}\E{\max\limits_{\xi}\sum\limits_{x \in N_\lambda}f_\xi(x)} \le \left(1+\frac{2C}{d-1}\right)\epsilon\]
for all $\epsilon > 0$.  Which concludes the proof.
\end{proof}

\subsection{The sum of a horofunction and the distance function}

Recall that given a boundary horofunction $\xi$ we have defined $f_\xi(x) = \xi(x) + d(o,x)$.  In this section we analize the level sets of $f_\xi$ to obtain a result needed in the proof of Theorem \ref{sumtheorem} above.

We recall that the upper half plane model of the hyperbolic plane is obtained identifying $\H^2$ with $\lbrace x+iy \in \C: y > 0\rbrace$ with the metric $\frac{1}{y^2}(dx^2+dy^2)$.   In this model we will set the base point $o = i$.  We will use explicit formulas for the distance function and horofunctions in this model, as well as the correspondence between the horofunction boundary and points on the extended real line (see for example \cite[Excersices 2.2, 6.10, 6.11]{bonahon}).

\begin{lemma}\label{ellipselemma}
In the upper half plane model of the hyperbolic plane the function $f_\xi$ asociated to the boundary point at $\infty$ is given by  
\[f(z) = 2\log\left(\frac{|z-i| + |z+i|}{2}\right).\]

In particular $f$ extends to all of $\C$ as a continuous function whose level sets are ellipses with foci at $\pm i$.
\end{lemma}
\begin{proof}
  The proof is by direct calculation.   The horofunction asociated to the boundary point at $\infty$ is $\xi(x+iy) = \log(y)$.  The distance $d(o,x+iy)$ can be calculated explicitely and is given by
  \[d(o,x+iy) = 2\log\left(\frac{|z-i|+|z+i|}{2\sqrt{y}}\right).\]
\end{proof}

The above calculation allows us to estimate the angular size of the level sets of $f_\xi$ as viewed from $o$.
\begin{lemma}\label{conesetlemma}
 For each $r > 0$ there exists $\theta_r$ such that for all boundary horofunctions $\xi$ on $\H^d$ there exists a unit tangent vector $v_\xi$ at $o$ such that the set $\lbrace f_\xi > r\rbrace$ is contained in the cone $D_{\theta_r}(v_\xi)$.  Futhermore, $\theta_r \approx e^{-r/2}$ when $r \to +\infty$
\end{lemma}
\begin{proof}
 Given $\xi$ there is a unique unit speed geodesic $\alpha(t) = \exp_o(tv)$ such that $\xi(\alpha(t)) = t$.   The function $f_\xi$ is invariant under all rotations in $\H^d$ fixing the points of $\alpha(t)$.  Hence it suffices to prove the result in $\H^2$ (by considering the planes containing $\alpha$).

 In the upper half plane model of $\H^2$ we may assume that the geodesic $\alpha(t)$ discussed above is $\alpha(t) = e^t i$.  And therefore that $\xi(x+iy) = \log(y)$.  By Lemma \ref{ellipselemma} one has 
 \[f_\xi(z) = f(z) = 2\log\left(\frac{|z-i|+|z+i|}{2}\right).\]
 
 Let $x_r > 0$ be such that $f(x_r) = r$.   It suffices to calculate the angle $\theta_r$ at $o$ between the geodesic ray $\alpha$ and the geodesic ray $\beta$ starting at $o$ whose endpoint is $x_r$.
 
 For this purpose we use the conformal transformation $z \mapsto \frac{z-i}{z+i}$ which maps the upper half plane to the unit disk.  Notice that $\alpha$ goes to the segment $[0,1]$ under this transformation.  On the other hand $\beta$ goes to another radius of the unit disk.  Hence, the angle $\theta_r$ is the absolute value of the smallest argument of $\frac{x_r-i}{x_r+i}$ from which one obtains
 \[\theta_r = 2\arctan\left(\frac{1}{x_r}\right).\]
 
 To conclude the proof one calculates from the equation $f(x_r) = r$ obtaining
 \[x_r = \sqrt{e^r-1}.\]
\end{proof}

\begin{figure}
\centering
\includegraphics[scale=1.2]{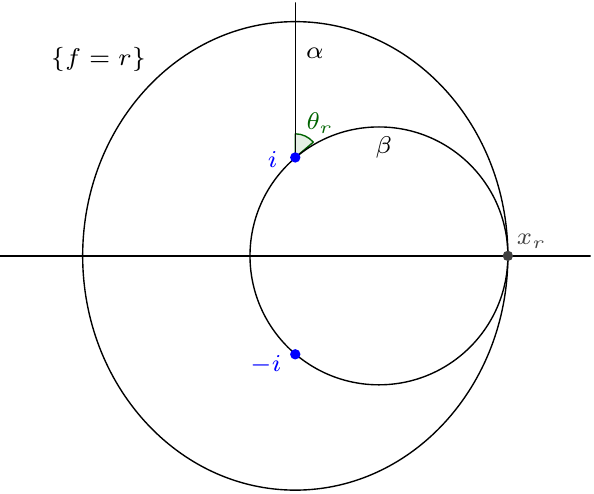}
\caption{\label{conesetfigure}Illustration of the proof of Lemma \ref{conesetlemma}.}
\end{figure}

\subsection{Poisson point processes on the sphere}

We consider the unit sphere $S^{d-1}$ in $\R^d$ with its normalized volume measure $m$.  To complete the proof of Theorem \ref{sumtheorem} we need the following application of the Vapnik-Chervonenkis inequality (see \cite{vapnik-chervonenkis}) to Poisson point processes on $S^{d-1}$.

\begin{lemma}\label{vapniklemma}
Suppose that for each $\mu > 0$ one has a Poisson point process $A_\mu$ with intensity $\mu$ on $S^{d-1}$, and for each $\mu$ a radius $\alpha_\mu$ is chosen such that $\alpha_\mu^{4(d-1)}\mu$ remains bounded away from $0$ when $\mu \to +\infty$.   Then 
\[\E{\max\limits_{p \in S^{d-1}}|A_\mu \cap B_{\alpha_\mu}(p)|} \approx \alpha_\mu^{d-1}\mu\]
when $\mu \to +\infty$. 
\end{lemma}
\begin{proof}
The fact that $\mu \alpha_\mu^{d-1} \lesssim \E{\max\limits_{p}|A_\mu \cap B_{\alpha_\mu}(p)|}$ for all $\mu$ large enough is trivial since one can pick a fixed ball of radius $\alpha_\mu$ for each $\mu$ and the number of points in it bounds the maximum from below.

To prove the upper bound notice that, by the Cauchy-Schwarz inequality, for all $\mu$ one has
\begin{align*}
\E{\max\limits_p|A_\mu \cap B_{\alpha_\mu}(p)|} &\le \E{|A_\mu|^2}^{\frac{1}{2}}\E{\max\limits_p\frac{|A_\mu \cap B_{\alpha_\mu}(p)|^2}{|A_\mu|^2}}^{\frac{1}{2}}
\\ &\lesssim \mu \E{\max\limits_p\frac{|A_\mu \cap B_{\alpha_\mu}(p)|^2}{|A_\mu|^2}}^{\frac{1}{2}}.
\\ &= \mu \E{X_\mu^2}^{\frac{1}{2}},
\end{align*}
where $X_\mu$ is the maximal proportion of points of $A_\mu$ to be found in a ball of radius $\alpha_\mu$.

To bound the second moment of $X_\mu$ we use the fact that, conditioned on $|A_\mu| = n$, the distribution of $A_\mu$ is that of $n$ i.i.d.\,uniform points.  Hence, the Vapnik-Chervonenkis inequality implies
\[\Pof{\left|X_\mu - m(B_{\alpha_\mu})\right|>t {\bigg\vert} |A_\mu|=n} \le Ce^{-c_0nt^2}\]
for all $t \ge 0$, where $c_0$ and $C$ are positive constants.

Using this, the explicit formula for $\E{e^{tX}}$ when $X$ is Poisson, and the inequality $1-e^{-x} \ge e^{-4c_0}x$ for $x \in [0,4c_0]$, one obtains that
\[\Pof{\left|X_\mu - m(B_{\alpha_\mu}(p))\right|>t} \le \E{Ce^{-c_0|A_\mu|t^2}} = Ce^{-\mu(1 - e^{-c_0t^2})} \le Ce^{-c\mu t^2},\]
for all $t \in [0,2]$ where $c = e^{-4c_0}c_0$.

Since one has a Gaussian tail bound (notice that for $t \notin [0,2]$ the probability on the left hand side above is clearly $0$) one obtains that for some $C' > 0$ one has
\begin{align*}
\E{X_\mu^2}  &= m(B_{\alpha_\mu})^2 + \E{|X_\mu - m(B_{\alpha_\mu})|^2}
\\ &\le m(B_{\alpha_\mu})^2 + \frac{C'}{\sqrt{\mu}}
\\ &= m(B_{\alpha_\mu})^2\left(1 + \frac{C'}{\sqrt{m(B_{\alpha_\mu})^4\mu}}\right), 
\end{align*}
from which the desired upper bound follows immediately.
\end{proof}

\part{Dimension drop phenomena\label{dimensiondroppart}}

The term \emph{dimension drop} refers to the fact, that in many situations, the distribution of the first exit point or limit point on a boundary at infinity, associated to a random walk has been observed to have smaller dimension than may be expected.

As an example consider a Brownian motion \(X_t\) starting at an interior point of a simply connected domain bounded by a Jordan curve in the plane, and let \(\tau\) be the first time \(X_t\) hits the boundary curve.  The distribution of \(X_{\tau}\) is always one dimensional as shown by Makarov in \cite{makarov}.  Hence, when the Jordan curve has dimension larger than one, the dimension drop phenomena occurs.

A second example is given by the simple random walk on certain rooted Galton-Watson trees.  In this case the limit point of the walk on the boundary at infinity (which is the set of infinite rays starting at the root with a natural metric) also exhibits the dimension drop phenomena for almost every realization of the random tree (see \cite{lyons-permantle-peres1995}, and also \cite{rousselin} and the references therein).

The purpose of this part of the article is to establish that, conditioned on the Poisson process, the limit point on the visual boundary of low intensity hyperbolic Poisson-Delaunay random walks exhibits the dimension drop phenomena almost surely.

With the same techinique we will also give some examples of co-compact Fuchsian groups for which the limit point at infinity of the simple random walk exhibits the dimension drop phenomena.  It might be the case that this type of dimension drop occurs for all co-compact Fuchsian groups, but our proof does not adapt easily to show this.

\section{Tools for proving dimension drop}

In this section we prove two results (Lemma \ref{dimensionupperboundlemma} and Lemma \ref{angleconvergencelemma}) which allow one to show that dimension drop occurs for certain measures on the boundary of hyperbolic space.

Recall that the dimension of a probability measure \(\nu\) is the smallest exponent \(\alpha\) such that there exists a set of \(\nu\) full measure with dimension less than \(\alpha\).    Equivalently it is the smallest exponent such that for \(\nu\) almost every $x$ one has
\[\liminf\limits_{r \to 0}\frac{\log(\nu(B_r(x)))}{\log(r)} \le \alpha.\]

\subsection{Dimension upper bound}

In this subsection we will give a general result which is useful to bound the dimension of the distribution of a limit of random variables from above.

We assume fixed in this subsection a complete separable metric space, we use $d(x,y)$ to denote the distance between two points \(x,y\), and \(B_r(x)\) to denote the open ball of radius \(r\) centered at \(x\).

Assume one has a sequence of random variables \(X_n\) taking values in the given metric space and converging almost surely to a random variable \(X\) when \(n\to +\infty\).  Let \(\nu_n\) be the distribution of \(X_n\) and \(\nu\) that of \(X\).

The lemma below allows one to transfer information on the measures \(\nu_n\) to estimate the dimension of \(\nu\) from above in certain circumstances.

Elementary examples where the lemma is applicable are obtained by setting \[X_n = \sum\limits_{k \le n}\sigma_k 2^{-k}\text{ or }X_n = \sum\limits_{k \le n}2\sigma_k 3^{-k}\] where the \(\sigma_k\) are i.i.d. with \(\P{\sigma_k=0}=\P{\sigma_k=1}=1/2\).  In these examples, taking \(r_n = 2^{-n}\) and \(r_n = 3^{-n}\) respectively, the lemma gives the optimal upper bounds for the dimensions of the distribution of the limit distributions (which are \(1\) and \(\log(2)/\log(3)\) respectively).

\begin{lemma}[Dimension upper bound]\label{dimensionupperboundlemma}
  Assume there exists a positive exponent \(\alpha\), positive random variables \(C\) and \(D \le 1\), and a deterministic sequence of positive radii \(r_n, n=1,2,\ldots\) which converges to \(0\), such that almost surely
  \[d(X_n,X) \le Cr_n\text{ and }\nu_n(B_{r_n}(X_n)) \ge Dr_n^\alpha\text{ for all }n.\]   

  Then the dimension of \(\nu\) is at most \(\alpha\). 
\end{lemma}

\begin{proof}
It suffices to prove the result in the case where the random variable \(C\) is constant.  Assuming this the general case follows by noticing that for each constant \(K\) the dimension of the distribution of \(X\) conditioned on the event \(C \le K\) is at most \(\alpha\).  Hence, there exists a set of \(\nu\) measure \(\Pof{C \le K}\) with dimension at most \(\alpha\).  Since this is valid for all \(K\) one obtains that the dimension of \(\nu\) is at most \(\alpha\) as claimed.

We will now prove the result in the case where \(C\) is constant.

For this purpose consider random variables \(X_n',n=1,2,\ldots\) and \(X'\) with the same joint distribution as \(X_n,n=1,2,\ldots\) and \(X\), and independent from them.  

Using the independence of these two sets of random variables, and the fact that \(d(X,X') \le d(X_n,X_n') +2Cr_n\), one obtains that almost surely
\begin{align*}\nu(B_{(1+2C)r_n}(X)) &= \Pof{X' \in B_{(1+2C)r_n}(X)|X}
\\ &\ge \Pof{X_n' \in B_{r_n}(X_n)|X}
\\ &= \E{\Pof{X_n' \in B_{r_n}(X_n)|X_n,X}|X}
\\ &= \E{\nu_n(B_{r_n}(X_n))|X}
\\ &\ge \E{D|X}r_n^\alpha.
\end{align*}

Hence, setting \(R_n = (1+2C)r_n\), for \(\nu\) almost every \(x\) there exists a positive constant $\epsilon_x$ such that
\[\nu(B_{R_n}(x)) \ge \epsilon_x R_n^\alpha\]
for all $n$.

This shows that the dimension of \(\nu\) is at most \(\alpha\) as claimed.
\end{proof}

\subsection{Speed of angular convergence}

We will now prove a geometric result which allows one to apply Lemma \ref{dimensionupperboundlemma} to hyperbolic random walks.

\begin{lemma}[Speed of angular convergence]\label{angleconvergencelemma}
Let \((x_n)_{n \ge 1}\) be a sequence in \(\H^d\) and \((\theta_n)_{n \ge 1}\) the corresponding sequence of projections onto the unit tangent sphere at a base point \(o\).   If \(d(x_n,x_{n+1}) = o(n)\) and \(d(x_0,x_n) = \ell n + o(n)\) for some \(\ell > 0\) when \(n \to +\infty\) then the limit \(\theta_\infty = \lim \theta_n\) exists and furthermore for all \(\ell' < \ell\) one has \(d(\theta_n,\theta_\infty) <  e^{-\ell' n}\) for all \(n\) large enough.
\end{lemma}
\begin{proof}
 Given \(\ell' < \ell\) and \(\epsilon > 0\) one has \(d(x_0,x_n) \ge \ell' n\) and \(d(x_n,x_{n+1}) \le \epsilon n\) for all \(n\) large enough.
 
 Hence the geodesic joining \(x_n\) and \(x_{n+1}\) does not intersect the ball of radius \((\ell'-\epsilon)n\) centered at the base point.

 The hyperbolic metric in polar coordinates is given by \(dr^2 + \sinh(r) d\theta^2\) where \(d\theta^2\) is the metric on the unit tangent sphere at the base point.   Since \(\sinh(r) \ge ce^{r}\) for some \(c > 0\) and all \(r\) large enough one obtains
 \[ce^{(\ell'-\epsilon)n}d(\theta_n,\theta_{n+1}) \le d(x_n,x_{n+1}) \le \epsilon n.\]
 
 Hence one obtains \(d(\theta_n,\theta_{n+1}) \le e^{-(\ell'-2\epsilon)n}\) for all \(n\) large enough.  In particular \(\theta_\infty = \lim \theta_n\) exists, and 
 \[d(\theta_n,\theta_\infty) \le \sum\limits_{k \ge n}d(\theta_k,\theta_{k+1}) \le \frac{e^{-(\ell'-2\epsilon)n}}{1 - e^{-(\ell'-2\epsilon)}} \le e^{-(\ell'-3\epsilon)n}\]
 for all \(n\) large enough.
 
 Since this holds for all \(\epsilon > 0\) and \(\ell' < \ell\) this concludes the proof.
\end{proof}

\section{Dimension drop for some co-compact Fuchsian groups}

Given natural numbers \(p,q \ge 3\) satisfying \(\frac{1}{p}+\frac{1}{q} < \frac{1}{2}\) there exists an essentially unique tessellation of the hyperbolic plane by regular \(p\)-gons with \(q\) meeting at each vertex.  We fix from now on, for each suitable choice of \(p\) and \(q\), such a tesselation in the upper half plane model containing the base point \(i\) as a vertex.  Let \(N_{p,q}\) denote the set vertices which are neighbors of \(i\) in the tesselation, \(r_{p,q}\) denote the length of the sides of the polygons in the tesselation (which is also the distance from each point in \(N_{p,q}\) to \(i\)).

\begin{figure}
\centering
\includegraphics[scale=0.45]{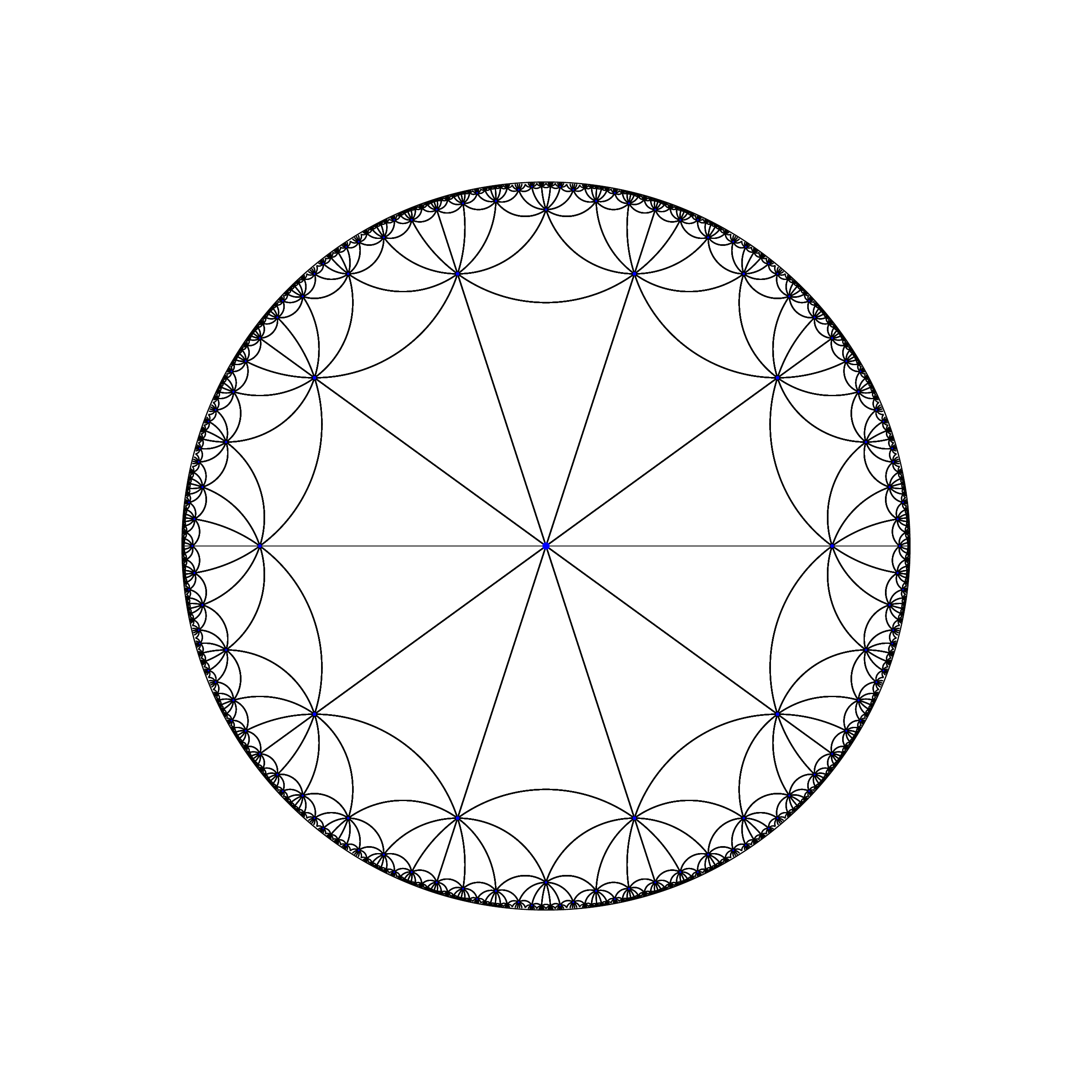}
\caption{\label{tessellation3-10figure}A tesselation by regular \(p\)-gons with \(q\)-meeting at each vertex in the Poincaré disk model (here \(p=3\) and \(q=10\)).}
\end{figure}

For each suitable \(p,q\) one may consider the simple random walk on the vertices of the tessellation starting at \(i\).  Let \(\ell_{p,q}\) be the speed of this random walk.  

The random walk may be realized in the form \(x_n = g_1\cdots g_n(i)\) where the \(g_i\) are i.i.d.\ and chosen uniformly from a finite symmetric generator of a co-compact Fuchsian group.  From Furstenberg's theory of Lyapunov exponents (see \cite{gruet} and Section \ref{rightangledsection}), one obtains that \(\ell_{p,q}\) is positive and almost surely constant.   Hence by Lemma \ref{angleconvergencelemma} letting \(\theta_n\) be the projection of \(x_n\) onto the unit tangent sphere at \(i\) (or equivalently onto the extended real line equiped with the visual metric at \(i\)), there exists a limit \(\theta_\infty = \lim \theta_n\) almost surely.  Let \(\nu_{p,q}\) denote the distribution of the limit point, we call this the exit measure (or harmonic measure) of the random walk.

We will prove that the dimension drop phenomena occurs for the exit measure of these simple random walks when \(q\) is large (the number of sides \(p\) plays no role in our estimates, in particular we show dimension drop for regular triangulations with sufficiently many triangles per vertex).

\begin{theorem}[Dimension drop for some co-compact Fuchsian groups]\label{fuchsiandimensiondrop}
The dimension of the exit measure of the simple random walk on the tessellation of the hyperbolic plane by regular \(p\)-gons with \(q\) meeting at each vertex satisfies the following estimate uniformly in \(p\):
\[\limsup\limits_{q \to +\infty}\text{dim}(\nu_{p,q}) \le \frac{1}{2}.\] 
\end{theorem}

The proof depends on obtaining good estimates for \(\ell_{p,q}\) and does not seem to extend easily to all co-compact Fuchsian groups.

\begin{lemma}\label{fuchsianspeedestimate}
The speed \(\ell_{p,q}\) satisfies 
\[\ell_{p,q} = 2\log(q) + O(\log(\log(q)))\]
uniformly in \(p\) when \(q \to +\infty\).
\end{lemma}
\begin{proof}[Proof of Lemma \ref{fuchsianspeedestimate}]
Consider a triangle joining a vertex, the center, and the midpoint of a side, of a regular \(p\)-gon with interior angle \(2\pi/q\).  The interior angles of this triangle are \(\pi/p,\pi/q\) and \(\pi/2\), and the side opposite to the angle \(\pi/p\) has length \(r_{p,q}/2\).  By the hyperbolic law of cosines one obtains
\[r_{p,q} = 2\acosh\left(\frac{\cos(\pi/p)}{\sin(\pi/q)}\right).\]

We set \(r_{\infty,q} = \lim\limits_{p \to +\infty}r_{p,q}\).  Since \(r_{\infty,q}-r_{3,q}\) is uniformly bounded one obtains
\[r_{p,q} = 2\log(q) + O(1)\]
uniformly in \(p\) when \(q \to +\infty\).

Next observe that by postivity of the speed one obtains from the Furstenberg type formula (Theorem \ref{furstenbergtheorem}) and Proposition \ref{corollaryboundary} that there exists a random boundary horofunction \(\xi_{p,q}\) such that
\[0 \le r_{p,q} - \ell_{p,q} = \E{\frac{1}{q}\sum\limits_{x \in N_{p,q}}d(i,x)+\xi_{p,q}(x)}.\]

Hence it suffices to show that
\[\max\limits_{\xi}\frac{1}{q}\sum\limits_{x \in N_{p,q}}f_\xi(x) = O(\log(\log(q)))\]
uniformly in \(p\) when \(q \to +\infty\) where \(f_\xi(x) = \xi(x) + d(i,x)\) and the maximum is over all boundary horofunctions.

By Lemma \ref{conesetlemma} the set of points where \(f_\xi\) is larger than \(2\log(\log(q))\) is contained in a cone with angle \(C\log(q)^{-1}\) for some constant \(C\) independent of \(q\).  Hence there are at most \(O(q/\log(q))\) points of \(N_{p,q}\) in this set.  Bounding the value of \(f_\xi\) at those points by \(2r_{p,q} = 4\log(q)+O(1)\) one obtains
\[\max\limits_{\xi}\frac{1}{q}\sum\limits_{x \in N_{p,q}}f_\xi(x) \le \frac{1}{q}O(q/\log(q))(4\log(q)+O(1)) + 2\log(\log(q)) = O(\log(\log(q)))\]
which establishes the lemma.
\end{proof}

We will now prove the main theorem in this section.  The proof below may be simplified somewhat by using the expression for the dimension of the harmonic measure on a Fuchsian group given for example in \cite{Tanaka}.  Instead we will give an argument closer to that which will be applied later on to study establish dimension drop for hyperbolic Poisson-Delaunay random walks.

\begin{proof}[Proof of Theorem \ref{fuchsiandimensiondrop}]
As before, let \(\theta_n\) be the projection of \(x_n\) onto the unit tangent sphere at \(i\), and \(\theta_\infty = \lim\theta_n\).   Applying Lemma \ref{angleconvergencelemma} one obtains, given \(\ell' < \ell\) a positive random variable \(C\) such that
\[d(\theta_n,\theta_\infty) \le Ce^{-\ell' n}\text{ for all }n.\]

Recall that the asymptotic entropy of the random walk on the tessellation is defined by
\[h_{p,q} = \lim\limits_{n \to +\infty}-\frac{1}{n}\log(p^n(x_0,x_n))\]
where \(p^n(x,y)\) is the \(n\)-step transtition probability between the vertices \(x\) and \(y\).

By subadditivity one has the estimate \(h_{p,q} \le \log(q)\) for all \(p\) and \(q\).  In fact, since there are \(q\) neighbors at each step one has \(p^n(x_0,x_n) \ge q^{-n}\) almost surely, but we will ignore this observation in order to illustrate the argument to be used later on for Poisson-Delaunay random walks.

Letting \(\nu_n\) be the distribution of \(\theta_n\) notice that, given \(h' > h_{p,q}\) there exists a positive random variable \(D \le 1\) such that
\[\nu_n(B_{r_n}(\theta_n)) \ge De^{-h' n}\text{ for all }n,\]
where \(r_n = e^{-\ell' n}\).

Hence by Lemma \ref{dimensionupperboundlemma} one obtains that the dimension of \(\nu_{p,q}\) is at most \(h'/\ell'\).   Since this holds for all \(h' > h_{p,q}\) and \(\ell' < \ell\) one has the following inequality (in fact equality holds as is shown in \cite{Tanaka}):
\[\text{dim}(\nu_{p,q}) \le \frac{h}{\ell}.\]

Applying Lemma \ref{fuchsianspeedestimate} one obtains
\[\text{dim}(\nu_{p,q}) \le \frac{h_{p,q}}{\ell_{p,q}} \le \frac{\log(q)}{2\log(q)+ O(\log(\log(q)))} = \frac{1}{2} + o(1)\]
which concludes the proof.
\end{proof}

\section{Dimension drop for low intensity hyperbolic Poisson-Delaunay random walks}

We return from now on to the notation of Part \ref{speedasymptoticspart}.  In particular let \(P_\lambda\) be a Poisson point process in \(\H^d\), \(o\) a fixed base point.

Recall that the speed \(\ell_\lambda\) is defined as
\[\ell_\lambda = \lim\limits_{n \to +\infty}\frac{1}{n}d(x_{0,\lambda},x_{n,\lambda})\]
where, conditioned on \(P_\lambda\), $(x_{n,\lambda}, n \geq 0)$ is a simple random walk on the Delaunay graph of \(P_\lambda \cup \lbrace o\rbrace\) starting at \(o\).

By the results of \cite{paquette} one obtains that both \(\ell_\lambda\) and the corresponding speed measured in the graph distance are almost surely positive (we have also given an independent proof of this for all small enough \(\lambda\)).

Recall also that the asymptotic entropy \(h_\lambda\) is the limit
\[h_\lambda = \lim_{n \to \infty} -\frac{1}{n}\log(p^n(x_{0,\lambda}, x_{n,\lambda}))\]
where \(p^n(x,y)\) denotes the \(n\)-step transition probability between \(x,y \in P_\lambda \cup \lbrace o\rbrace\) conditioned on \(P_\lambda\).  This limit is guaranteed to exist and be positive almost surely for the Poisson-Delaunay graph by the positivity of the graph speed (see \cite[Lemma 5.1]{carrascopiaggio2016}).

By distance stationarity \(d(x_{n,\lambda},x_{n+1,\lambda}) = o(n)\) almost surely.   Hence, letting \(\theta_{n,\lambda}\) denote the projection of \(x_{n,\lambda}\) onto the unit tangent sphere at \(o\) one obtains that the limit \(\theta_{\infty,\lambda} = \lim\theta_{n,\lambda}\) exists almost surely by Lemma \ref{angleconvergencelemma}.  

Notice that by rotational invariance of the Poisson point process the distribution of \(\theta_{\infty,\lambda}\) is uniform on the unit tangent sphere at \(o\).   However, we will show that the distribution \(\nu_\lambda\) of \(\theta_{\infty,\lambda}\) conditioned on \(P_\lambda\) typically has dimension smaller than \(d-1\).

\subsection{Dimension upper bound}

\begin{lemma}[Dimension upper bound]\label{dimupperboundlemma}
 For each \(\lambda\), the speed \(\ell_\lambda\), asymptotic entropy \(h_\lambda\), and dimension \(\text{dim}(\nu_\lambda)\) are almost surely constant and \(\text{dim}(\nu_\lambda) \le h_\lambda/\ell_\lambda\).
\end{lemma}
\begin{proof}
The first part of the statement follows immediately from Theorem \ref{ergodicity} (see also Corollary \ref{ergodicitycorollary}).  We will now prove the claimed inequality.

Suppose \(\lambda\) is fixed in what follows, and set \(\ell = \ell_\lambda\), \(x_n = x_{n,\lambda}\), and \(h = h_{\lambda}\).

By Lemma \ref{angleconvergencelemma}, given \(\ell' < \ell\), there exists a positive random variable \(C\) such that
 \[d(\theta_{n,\lambda},\theta_{\infty,\lambda}) \le C e^{-\ell' n}\text{ for all }n.\]

On the other hand, by the definition of the asymptotic entropy \(h\), one has that for any \(h' > h\) there exists a positive random variable \(D\) (which one may choose to be bounded from above by \(1\)) such that
\[p^n(x_0,x_n) \ge De^{-h' n}\text{ for all }n,\]
where \(p^n(x,y)\) is the \(n\)-th step transition probability for the simple random walk on the Delaunay graph of \(P_\lambda \cup \lbrace o\rbrace\) conditioned on \(P_\lambda\).  

Set \(r_n = e^{-\ell' n}\) and notice that if \(\nu_n\) is the distribution of \(\theta_{n,\lambda}\) then, trivially, \(\nu_n(B_{r_n}(\theta_{n,\lambda})) \ge p^n(x_0,x_n)\) (since \(\nu_n\) has a point mass of at least this amount at \(\theta_{n,\lambda}\)).  Hence, applying Lemma \ref{dimensionupperboundlemma} one obtains
\[\text{dim}(\nu_\lambda) \le \frac{h'}{\ell'}.\]

Since this is valid for all \(\ell' < \ell\) and \(h' > h\), the proof is complete.
\end{proof}

\subsection{Dimension drop}

We will now prove the main result of this section, establishing dimension drop for low intensity Poisson-Delaunay random walks.

\begin{theorem}[Dimension drop for low intensity hyperbolic Poisson-Delaunay random walks]\label{PDdimensiondrop}
In the notation above one has \(\limsup_{\lambda \to 0} \text{dim}(\nu_\lambda) \leq \frac{d-1}{2}\). 
\end{theorem}
\begin{proof}
By Lemma \ref{dimupperboundlemma} and Theorem \ref{speedestimatestheorem} one has
\[\limsup\limits_{\lambda \to 0} \text{dim}(\nu_\lambda) \leq \limsup\limits_{\lambda \to 0}\frac{h_\lambda}{\ell_\lambda} = \frac{d-1}{2}\limsup\limits_{\lambda \to 0}\frac{h_\lambda}{\log(\lambda^{-1})},\]
hence it suffices to estimate \(h_\lambda\) as \(\lambda \to 0\).

For this purpose notice that by stationarity under the degree biased measure one has
\[h_\lambda \le \Elambda{\deg(o)}\]
where \(\deg(o)\) is the number of neighbors of the base point.

Using Jensen's inequality applied to \(x \mapsto x\log(x)\) obtains
\[\Elambda{\deg(o)} = \E{\deg(o)}^{-1}\E{\deg(o)\log(\deg(o))} \le \log(\E{\deg(o)}).\]

Finally, by Lemma \ref{neighborlemma} one has that \(\E{\deg(o)}/\lambda^{-1}\) is bounded when \(\lambda \to 0\).  Hence one obtains
\[\limsup\limits_{\lambda \to 0}\frac{h_\lambda}{\log(\lambda^{-1})} \le 1\]
from which the theorem follows immediatly.
\end{proof}

\end{document}